\documentclass[11pt]{article}
\usepackage[utf8]{inputenc}
\usepackage[T1]{fontenc}
\usepackage{lmodern}
\usepackage{amsmath,amssymb,amsthm,mathtools,mathrsfs}
\usepackage{geometry}
\usepackage{hyperref}
\usepackage{enumitem}
\usepackage{microtype}
\usepackage{xurl}
\usepackage{booktabs}
\usepackage{float}
\usepackage{tikz}
\usepackage{pgfplots}
\usetikzlibrary{arrows.meta,positioning,calc}
\pgfplotsset{compat=1.18}
\hypersetup{colorlinks=true,linkcolor=blue,citecolor=blue,urlcolor=blue}

\newtheorem{theorem}{Theorem}[section]
\newtheorem{lemma}[theorem]{Lemma}
\newtheorem{proposition}[theorem]{Proposition}
\newtheorem{corollary}[theorem]{Corollary}
\theoremstyle{definition}
\newtheorem{definition}[theorem]{Definition}
\newtheorem{remark}[theorem]{Remark}

\newcommand{\Q}{\mathbb{Q}}
\newcommand{\Z}{\mathbb{Z}}

\newcommand{\Ftwoone}{{}_2F_1}
\newcommand{\FthreeTwo}{{}_3F_2}
\newcommand{\PCF}{\operatorname{PCF}}
\newcommand{\Sym}{\operatorname{Sym}}
\newcommand{\GL}{\operatorname{GL}}

\newcommand{\dd}{\mathrm{d}}

\newcommand{\sig}{\sigma}

\title{Order-3 $\pi$-formulas, Ap\'ery-like kernels, and Clausen functoriality for Conservative Matrix Fields}
\author{Alex Shvets}
\date{}

\begin{document}
\maketitle

\begin{abstract}
Raz, Shalyt, Leibtag, Kalisch, Weinbaum, Hadad, and Kaminer recently showed that formulas for $\pi$ can be organized by canonical polynomial recurrences and partially unified by a rank-$2$ Conservative Matrix Field (CMF). We prove that each order-$3$ recurrence explicitly printed in the public Appendix~B.6 of their paper is a shifted summation lift of an explicit order-$2$ kernel, and identify all three kernels: the two $\pi$-kernels are explicit rescalings of the sporadic Ap\'ery-like sequences $A036917$ and $A002895$ (Domb numbers, case~$(\alpha)$), while the Catalan kernel is a hypergeometric twist of the Gauss-square coefficient sequence at $(a,b,c)=(\tfrac12,1,\tfrac32)$. We place these kernels in a unified $\Sym^2$ framework: the first $\pi$-kernel and the Catalan kernel come directly from Gauss-square coefficient sequences, while the Domb kernel is recovered by recasting the classical degree-$3$ Belyi pullback $\phi(x)=108x^2/(1-4x)^3$ and the associated algebraic twist in CMF language. We write an explicit square-gauge matrix for the Gauss CMF, formulate the standard pullback--twist transport in CMF terms, and show that for rank-$2$ objects it is compatible with $\Sym^2$. We further prove an inverse classification: for a fixed $\Sym^2$-type Riemann scheme, the one-parameter family of Fuchsian operators contains a unique $\Sym^2(\mathrm{Gauss})$ point, cut out by the closed-form condition $\lambda_0=2\gamma_1\gamma_2(1-2\alpha)$ on the accessory parameter. Finally, a Belyi-pullback scan over $5040$ configurations produces $11$ additional integer sequences of the form $[x^n]\lambda^n\,{}_2F_1(a,b;c;\phi(x))^2$; we prove their integrality and place them in the same $\Sym^2$-pullback framework.
\end{abstract}

\section{Introduction}
The recent preprint of Raz, Shalyt, Leibtag, Kalisch, Weinbaum, Hadad, and Kaminer develops an automated pipeline for harvesting, validating, canonicalizing, and partially unifying formulas for mathematical constants, with the flagship case study being $\pi$ \cite{Raz}. Applied to $385$ validated formulas for $\pi$, their procedure produces $153$ canonical polynomial recurrences: $149$ of order $2$ and $4$ of order $3$ \cite[\S3.3]{Raz}. The same work exhibits a concrete rank-$2$ three-dimensional Conservative Matrix Field (CMF) for $\pi$ and proves that $81$ of the $153$ canonical forms reside in that CMF \cite[Table~4]{Raz}.

The appearance of order-$3$ canonical forms raises a natural structural question. Are these order-$3$ objects genuinely outside rank-$2$ CMF geometry, or can some of them be decomposed into lower-order kernels together with external summation? The public Appendix~B.6 of \cite{Raz} prints two explicit order-$3$ recurrences for $\pi$ and one for Catalan's constant. We apply a standard Ore-algebraic factorization criterion---right divisibility by $(S-1)$, equivalent to the vanishing of the coefficient sum---to these three printed recurrences and show that each is a shifted summation lift of an order-$2$ kernel.

Our second result identifies all three printed kernels. The first $\pi$-kernel is an explicit rescaling of the classical Ap\'ery-like sequence $A036917$ \cite{Delaygue,OEISA036917}; the second $\pi$-kernel is an explicit rescaling of the Domb numbers $A002895$ (case~$(\alpha)$) \cite{OEISA002895,Gorodetsky}; and the Catalan kernel is a hypergeometric twist of the coefficient sequence of ${}_2F_1(\tfrac12,1;\tfrac32;z)^2$. The connections between $A036917$, $A002895$ and series for $1/\pi$ are well established \cite{ChanVerrill,Cooper}; what is new is their identification as the specific summation kernels of the printed order-$3$ recurrences of \cite{Raz}. In particular, the two $\pi$-kernels are sporadic Ap\'ery-like sequences while the Catalan kernel is a non-sporadic Gauss-square point. The first $\pi$-kernel and the Catalan kernel are explicit rescalings of ${}_2F_1^2$ coefficient sequences; the second $\pi$-kernel (Domb) is related to the ${}_2F_1$-squared world only through an algebraic pullback.

Our third result is an explicit square-gauge formula for the Gauss hypergeometric CMF. Let
\[
f(a,b,c;z):={}_2F_1(a,b;c;z),\qquad g(a,b,c;z):=f(a,b,c;z)^2.
\]
We compute a rational matrix $\Phi(a,b,c;z)$ such that the CMF generated by $g$ in the basis $(g,\theta g,\theta^2 g)$ is gauge-equivalent to the square basis $(f^2,f\theta f,(\theta f)^2)$. In the pure Clausen regime this specializes to an explicit gauge relation between a Gauss $\Ftwoone$-CMF and a hypergeometric $\FthreeTwo$-CMF.

Our fourth result connects the ambient $\pi$-CMF to the classical theory. The third-order ODE for $g(z)={}_2F_1(a,b;c;z)^2$ was first derived by Chaundy \cite{Chaundy} and is explicit in Vid\=unas \cite[Eq.~(31)]{Vidunas}; the corresponding coefficient recurrence for general parameters has recently been obtained by Mao and Tian \cite{MaoTian}. We show that, when this ODE is interpreted as the differential component $M_\theta$ of the square Gauss CMF, it recovers the first printed $\pi$-kernel and, after a direct Gauss-square specialization, the printed Catalan kernel as well.

Our fifth result closes the Domb branch. We show that the Domb generating function equals $[1/(1-4x)]\cdot{}_2F_1(\tfrac16,\tfrac13;1;\phi(x))^2$ where $\phi(x)=108x^2/(1-4x)^3$ is a degree-$3$ Belyi map ramified only over $\{0,1,\infty\}$. We then prove a general pullback--twist functoriality theorem: composing a CMF basis with a rational pullback and a scalar twist produces a new CMF whose shift and differential generators are given by explicit formulas, and $\Sym^2$ commutes with this operation. This places all three printed kernels inside a single mechanism.

\begin{figure}[t]
\centering
\begin{tikzpicture}[scale=0.8, transform shape,
    node distance=10mm and 9mm,
    >=Latex,
    box/.style={draw, rounded corners=2pt, align=center, minimum width=26mm, minimum height=8mm, inner sep=2.5pt, fill=gray!4},
    kernel/.style={draw, rounded corners=2pt, align=center, minimum width=31mm, minimum height=8mm, inner sep=2.5pt, fill=blue!4},
    formula/.style={draw, rounded corners=2pt, align=center, minimum width=34mm, minimum height=8mm, inner sep=2.5pt, fill=green!4},
    openbox/.style={draw, rounded corners=2pt, dashed, align=center, minimum width=34mm, minimum height=8mm, inner sep=2.5pt, fill=orange!4},
    lab/.style={font=\footnotesize, inner sep=1pt}
]
\node[box] (gauss) {ambient\\ ${}_2F_1$-CMF};
\node[box, right=of gauss] (sym) {$\Sym^2$};
\node[box, right=of sym] (mth) {differential\\ component $M_\theta$};

\node[kernel, above right=11mm and 15mm of mth] (a036) {$A036917$ kernel\\ (first $\pi$)};
\node[kernel, below right=0mm and 15mm of mth] (catk) {Catalan kernel\\ (Gauss-square)};
\node[kernel, below right=11mm and 15mm of mth] (domb) {Domb kernel\\ $A002895$};

\node[formula, right=20mm of a036] (piform) {printed order-$3$\\ $\pi$ formula};
\node[formula, right=20mm of catk] (catform) {printed order-$3$\\ Catalan formula};
\node[formula, right=20mm of domb] (dombform) {printed order-$3$\\ $\pi$ formula};

\draw[->, thick] (gauss) -- node[lab, above] {} (sym);
\draw[->, thick] (sym) -- node[lab, above] {} (mth);

\draw[->, thick] (mth) -- node[lab, above] {coefficients} (a036);
\draw[->, thick] (mth) -- node[lab, above] {specialize} (catk);
\draw[->, thick] (sym.south east) .. controls +(14mm,-12mm) and +(-18mm,12mm) .. node[lab, below, pos=0.5] {Belyi $\phi$} (domb.west);

\draw[->, thick] (a036) -- node[lab, above] {summation} (piform);
\draw[->, thick] (catk) -- node[lab, above] {summation} (catform);
\draw[->, thick] (domb) -- node[lab, above] {summation} (dombform);
\end{tikzpicture}
\caption{Structural map of the three printed order-$3$ examples. The first $\pi$-kernel and the Catalan kernel arise from Gauss-square specialization after rescaling (solid arrows via $M_\theta$). The Domb kernel arises from ${}_2F_1(\tfrac16,\tfrac13;1;z)$ via the Belyi pullback $\phi(x)=108x^2/(1-4x)^3$ and twist $1/(1-4x)$ (solid arrow via pullback).}
\label{fig:chain}
\end{figure}
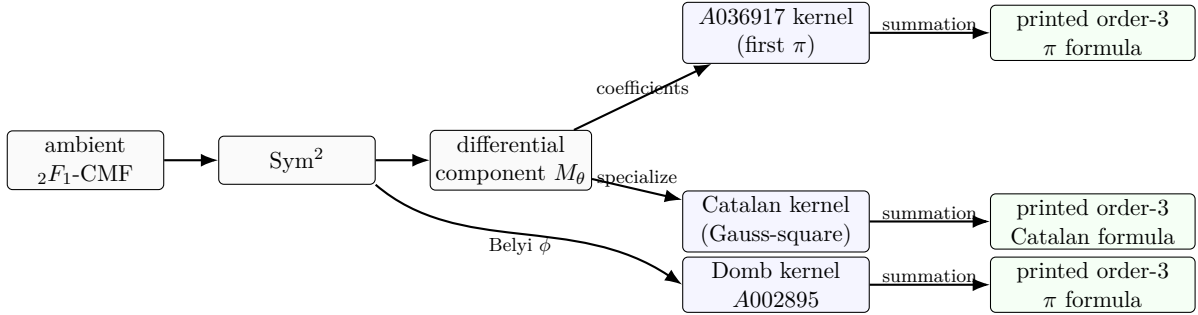

We stress three limitations. First, Raz et al.\ canonicalize formulas at the level of \emph{partial sums or convergents}; our order-$2$ kernels live at the level of summands or coefficient sequences. Thus our results refine, rather than contradict, the formula-level minimal orders reported in \cite{Raz}. Second, Raz et al.\ report four order-$3$ canonical forms for $\pi$, but only two explicit order-$3$ $\pi$-recurrences are printed in the public Appendix~B.6; the fourth order-$3$ $\pi$-canonical form is not explicit in the public materials used here. Third, our intrinsic derivation of $A036917$ uses the differential component $M_\theta$ of the ambient Gauss CMF. It remains open whether the same kernel appears as a \emph{trajectory} of the current rank-$2$ $\pi$-CMF in the sense of parameter shifts.

The paper is organized as follows. Section~\ref{sec:preliminaries} recalls the minimal amount of CMF and canonical-form formalism needed below. Section~\ref{sec:summation} proves the factorization criterion and applies it to the three printed order-$3$ recurrences from Appendix~B.6 of \cite{Raz}. Section~\ref{sec:aperylike-clausen} identifies all three kernels, proves the general square-gauge formula for Gauss CMFs, derives the intrinsic differential equation and coefficient recurrence of the square, records the pure Clausen corollary, establishes an inverse classification via the accessory parameter, and closes the Domb branch via a Belyi pullback and a pullback--twist functoriality theorem for CMFs. Section~\ref{sec:belyi-scan} reports the results of a Belyi-pullback scan over $5040$ parameter configurations, proves integrality for $11$ additional integer sequences, and places them in the $\Sym^2$-pullback framework. Section~\ref{sec:discussion} lists consequences and open problems.

\section{Preliminaries}
\label{sec:preliminaries}
\subsection{Conservative Matrix Fields and trajectories}
We use the modern CMF formalism of \cite{David,Weinbaum}. A CMF is a $1$-cocycle of $\Z^d$ with values in matrices of rational functions.

\begin{definition}
Let $K$ be a field and let $\sig_v$ denote the shift action of $v\in\Z^d$ on rational functions in $d$ variables. A \emph{Conservative Matrix Field} of dimension $d$ and rank $r$ over $K$ is a map
\[
\mathcal M\colon \Z^d\to \GL_r(K(\mathbf x)),\qquad v\mapsto \mathcal M_v,
\]
such that
\begin{equation}\label{eq:cocycle-intro}
\mathcal M_{v+w}=\mathcal M_v\,\sig_v(\mathcal M_w)\qquad (v,w\in\Z^d).
\end{equation}
\end{definition}

Following \cite{Weinbaum}, if $f$ is a D-finite function and $B$ is a basis of the finite-dimensional Ore-module generated by $f$, then the basis-change matrices relating $B$ and its parameter shifts form a CMF \cite[Theorem~3.1]{Weinbaum}. Moreover, changing the basis changes the CMF by a rational coboundary transformation \cite[Proposition~3.2]{Weinbaum}.

Given a basepoint $x$ and a direction $v\in\Z^d$, one obtains a one-dimensional object by sampling along the trajectory $x+n v$. The corresponding \emph{trajectory matrix} is
\[
T_{x,v}(n):=\mathcal M_v(x+n v).
\]
When a rational gauge transform converts $T_{x,v}(n)$ to companion form, the trajectory yields a linear recurrence in the step variable $n$. This is the bridge between CMFs and canonical polynomial recurrences used in \cite{Raz,Weinbaum}.

\subsection{Canonical forms and coboundary equivalence}
Raz et al.
associate to a formula for a constant its \emph{canonical form}: the minimal linear recurrence with polynomial coefficients satisfied by the corresponding sequence of rational approximants \cite[\S3.3 and Appendix~C.4]{Raz}. In the second-order case this is equivalently represented as a polynomial continued fraction (PCF) \cite[\S2.1 and Appendix~E]{Raz}. Throughout the paper we work in the Ore algebra
\[
\Q[n]\langle S\rangle,
\]
where $S$ is the forward shift and satisfies
\[
S\,p(n)=p(n+1)\,S \qquad (p\in\Q[n]).
\]
We also write $\sigma$ for the induced automorphism $\sigma(p)(n)=p(n+1)$.

Two polynomial matrix recurrences $A(n)$ and $B(n)$ are \emph{coboundary equivalent} if there exist a polynomial matrix $U(n)$ and nonzero polynomials $p_A(n),p_B(n)$ such that
\begin{equation}\label{eq:coboundary-poly}
p_A(n)A(n)U(n+1)=p_B(n)U(n)B(n).
\end{equation}
This is the notion used in the UMAPS algorithm of \cite[\S3.5 and Appendix~C.3]{Raz}.

\subsection{Symmetric square of a rank-$2$ CMF}
If a $2\times 2$ matrix acts on a row basis $(u,v)$ by right multiplication,
\[
[u,v]M=[u',v'],\qquad M=\begin{pmatrix}a&b\\ c&d\end{pmatrix},
\]
then the induced action on the symmetric-square basis $(u^2,uv,v^2)$ is given by
\begin{equation}\label{eq:sym2formula}
\Sym^2(M)=
\begin{pmatrix}
a^2 & ab & b^2\\
2ac & ad+bc & 2bd\\
c^2 & cd & d^2
\end{pmatrix}.
\end{equation}
This is the standard symmetric-square functor on $\GL_2$.

\subsection{Ap\'ery-like sequences}
We use the terminology of Ap\'ery-like sequences as in the work of Almkvist--Zudilin and Delaygue \cite{Delaygue}. The sequence
\[
A_n=\sum_{k=0}^n \binom{2k}{k}^2\binom{2n-2k}{n-k}^2
\]
is OEIS~A036917. It is one of the six sporadic Almkvist--Zudilin sequences (see \cite{Delaygue} and \cite{OEISA036917}); its generating function is the square of a Zagier generating function \cite{Gorodetsky}. We shall use this identification in Section~\ref{sec:aperylike-clausen}.

\section{Summation decomposition of printed order-$3$ recurrences}\label{sec:summation}
We begin with a standard algebraic criterion. The lemma below is a direct consequence of right Euclidean division in the Ore algebra $\Q[n]\langle S\rangle$; we include a short proof for the reader's convenience.

\begin{lemma}[summation-lift criterion]\label{lem:factor-criterion}
Let
\[
L_3=A_0(n)+A_1(n)S+A_2(n)S^2+A_3(n)S^3\in\Q[n]\langle S\rangle.
\]
The following are equivalent:
\begin{enumerate}[label=\textup{(\roman*)}]
\item there exist a nonzero constant $c\in\Q$ and an order-$2$ operator
\[
L_2=B_0(n)+B_1(n)S+B_2(n)S^2
\]
such that
\begin{equation}\label{eq:summation-factor}
L_3=c\,\sigma\!\bigl(L_2(S-1)\bigr);
\end{equation}
\item the coefficient identity
\begin{equation}\label{eq:sumzero}
A_0(n)+A_1(n)+A_2(n)+A_3(n)=0
\end{equation}
holds identically.
\end{enumerate}
Whenever \eqref{eq:sumzero} holds, the coefficients of $L_2$ are uniquely recovered from
\begin{equation}\label{eq:BfromA}
B_0(n)=-\frac{A_0(n-1)}{c},\qquad
B_1(n)=\frac{A_2(n-1)+A_3(n-1)}{c},\qquad
B_2(n)=\frac{A_3(n-1)}{c}.
\end{equation}
\end{lemma}

\begin{proof}
We expand
\[
L_2(S-1)=B_0(S-1)+B_1S(S-1)+B_2S^2(S-1)
\]
in the Ore algebra $\Q[n]\langle S\rangle$:
\[
L_2(S-1)=-B_0+(B_0-B_1)S+(B_1-B_2)S^2+B_2S^3.
\]
Applying $\sigma$ gives
\[
\sigma\!\bigl(L_2(S-1)\bigr)
=-\sigma(B_0)+\bigl(\sigma(B_0)-\sigma(B_1)\bigr)S+\bigl(\sigma(B_1)-\sigma(B_2)\bigr)S^2+\sigma(B_2)S^3.
\]
Therefore \eqref{eq:summation-factor} is equivalent to the system
\[
A_0=-c\,\sigma(B_0),\qquad A_1=c\bigl(\sigma(B_0)-\sigma(B_1)\bigr),
\]
\[
A_2=c\bigl(\sigma(B_1)-\sigma(B_2)\bigr),\qquad A_3=c\,\sigma(B_2).
\]
Adding the four equalities yields \eqref{eq:sumzero}. Conversely, if \eqref{eq:sumzero} holds, then the formulas \eqref{eq:BfromA} solve the above system, proving the existence of $L_2$. Uniqueness follows because the system is triangular in $\sigma(B_0),\sigma(B_1),\sigma(B_2)$.
\end{proof}

\begin{remark}\label{rem:summation-lift}
If a sequence $(c_n)$ is annihilated by $L_2$ and $s_n:=\sum_{k=0}^n c_k$, then $(S-1)s_n=c_{n+1}$. Hence any factorization of the form \eqref{eq:summation-factor} says precisely that the order-$3$ recurrence $L_3$ is obtained from the order-$2$ kernel $L_2$ by one summation.
\end{remark}

We now apply Lemma~\ref{lem:factor-criterion} to the three order-$3$ recurrences explicitly printed in Appendix~B.6 of \cite[p.~21]{Raz}: two for $\pi$ and one shared by two formulas for Catalan's constant.

\begin{theorem}[factorization of the printed order-$3$ recurrences]\label{thm:three-printed-factor}
The three order-$3$ recurrences printed in Appendix~B.6 of \cite[p.~21]{Raz} satisfy the criterion \eqref{eq:sumzero} and therefore factor as shifted summation lifts of order-$2$ kernels.

More precisely:
\begin{enumerate}[label=\textup{(\roman*)}]
\item The first printed $\pi$-recurrence factors as
\[
L^{(\pi,1)}_3=\frac14\,\sigma\!\bigl(L^{(1)}_2(S-1)\bigr),
\]
with
\begin{equation}\label{eq:L2firstkernel}
L^{(1)}_2=(n+1)^4-n(2n+3)(2n^2+6n+5)S+4n(n+1)(n+2)^2S^2.
\end{equation}
For \(n\ge 1\), define \(u_n=4^n(n-1)!(n!)^2 c_n\).
Then the inflated recurrence is
\[
u_{n+1}=(2n+1)(2n^2+2n+1)u_n-4n^6u_{n-1}\qquad (n\ge 2),
\]
equivalently the inflated kernel is represented by the polynomial continued fraction
\begin{equation}\label{eq:firstkernelPCF}
\PCF\bigl((2n+1)(2n^2+2n+1),-4n^6\bigr).
\end{equation}

\item The second printed $\pi$-recurrence factors as
\[
L^{(\pi,2)}_3=\frac1{144}\,\sigma\!\bigl(L^{(2)}_2(S-1)\bigr),
\]
with
\begin{equation}\label{eq:L2secondkernel}
\begin{aligned}
L^{(2)}_2={}&(n+1)^3(3n+4)(3n+7) \\
&+(2n+3)(3n+1)(3n+7)(5n^2+15n+12)S \\
&+16(n+2)^3(3n+1)(3n+4)S^2.
\end{aligned}
\end{equation}

\item The printed order-$3$ Catalan recurrence factors as
\[
L^{(G)}_3=\frac14\,\sigma\!\bigl(L^{(G)}_2(S-1)\bigr),
\]
with
\begin{equation}\label{eq:L2catalan}
L^{(G)}_2=(n+1)(n+2)-4(n+2)^2S+(2n+5)^2S^2.
\end{equation}
Equivalently, in recurrence form,
\begin{equation}\label{eq:catalankernelrec}
(2n+5)^2 g_{n+2}-4(n+2)^2 g_{n+1}+(n+1)(n+2)g_n=0.
\end{equation}
\end{enumerate}
\end{theorem}

\begin{proof}
We treat the three printed recurrences one by one.

\smallskip
\noindent\textbf{(i) First printed $\pi$-recurrence.}
Appendix~B.6 of \cite[p.~21]{Raz} prints
\[
\begin{aligned}
0={}&\left(-4-8n-6n^2-2n^3-\frac{n^4}{4}\right)f(n) \\
&+\left(\frac{81}{4}+\frac{173n}{4}+\frac{65n^2}{2}+\frac{21n^3}{2}+\frac{5n^4}{4}\right)f(n+1) \\
&+\left(-\frac{137}{4}-\frac{297n}{4}-\frac{111n^2}{2}-\frac{35n^3}{2}-2n^4\right)f(n+2) \\
&+\left(18+39n+29n^2+9n^3+n^4\right)f(n+3).
\end{aligned}
\]
Hence
\[
\begin{aligned}
A_0(n)&=-\frac{(n+2)^4}{4},\\
A_1(n)&=\frac{5n^4+42n^3+130n^2+173n+81}{4},\\
A_2(n)&=-\frac{(n+1)(8n^3+62n^2+160n+137)}{4},\\
A_3(n)&=(n+1)(n+2)(n+3)^2.
\end{aligned}
\]
Expanding the sum gives
\[
4\bigl(A_0+A_1+A_2+A_3\bigr)
=-(n+2)^4+(5n^4+42n^3+130n^2+173n+81)
\]
\[
-(8n^4+70n^3+222n^2+297n+137)+(4n^4+36n^3+116n^2+156n+72)=0.
\]
Thus Lemma~\ref{lem:factor-criterion} applies with $c=\tfrac14$, and \eqref{eq:BfromA} yields
\[
B_0(n)=-(4)A_0(n-1)=(n+1)^4,
\]
\[
B_1(n)=4\bigl(A_2(n-1)+A_3(n-1)\bigr)=-n(2n+3)(2n^2+6n+5),
\]
\[
B_2(n)=4A_3(n-1)=4n(n+1)(n+2)^2,
\]
which is exactly \eqref{eq:L2firstkernel}.

For the inflation, write the kernel recurrence as
\[
(n+1)^4 c_n-n(2n+3)(2n^2+6n+5)c_{n+1}+4n(n+1)(n+2)^2c_{n+2}=0.
\]
Set $u_n=4^n(n-1)!(n!)^2c_n$ for $n\ge1$. Then
\[
\frac{u_{n+1}}{u_n}=4n(n+1)^2\,\frac{c_{n+1}}{c_n},
\qquad
\frac{u_{n-1}}{u_n}=\frac{1}{4(n-1)n^2}\,\frac{c_{n-1}}{c_n}.
\]
Substituting these relations into the kernel recurrence and simplifying gives
\[
u_{n+1}=(2n+1)(2n^2+2n+1)u_n-4n^6u_{n-1},
\]
which is \eqref{eq:firstkernelPCF}.

\smallskip
\noindent\textbf{(ii) Second printed $\pi$-recurrence.}
Appendix~B.6 of \cite[p.~21]{Raz} prints
\[
\begin{aligned}
0={}&\left(-\frac{35}{9}-\frac{26n}{3}-\frac{23n^2}{3}-\frac{121n^3}{36}-\frac{35n^4}{48}-\frac{n^5}{16}\right)f(n)\\
&+\left(-\frac{365}{9}-\frac{181n}{2}-\frac{1879n^2}{24}-\frac{1589n^3}{48}-\frac{55n^4}{8}-\frac{9n^5}{16}\right)f(n+1)\\
&+\left(-\frac{356}{9}-\frac{503n}{6}-\frac{1633n^2}{24}-\frac{1279n^3}{48}-\frac{81n^4}{16}-\frac{3n^5}{8}\right)f(n+2)\\
&+\left(84+183n+154n^2+\frac{568n^3}{9}+\frac{38n^4}{3}+n^5\right)f(n+3).
\end{aligned}
\]
Thus
\[
\begin{aligned}
A_0(n)&=-\frac{(n+2)^3(3n+7)(3n+10)}{144},\\
A_1(n)&=-\frac{(3n+10)(27n^4+240n^3+789n^2+1128n+584)}{144},\\
A_2(n)&=-\frac{(3n+4)(3n+8)(6n^3+57n^2+177n+178)}{144},\\
A_3(n)&=\frac{(n+3)^3(3n+4)(3n+7)}{9}.
\end{aligned}
\]
Expanding the sum gives
\[
144\bigl(A_0+A_1+A_2+A_3\bigr)
=-(9n^5+105n^4+484n^3+1104n^2+1248n+560)
\]
\[
-(81n^5+990n^4+4767n^3+11274n^2+13032n+5840)
\]
\[
-(54n^5+729n^4+3837n^3+9798n^2+12072n+5696)
\]
\[
+(144n^5+1824n^4+9088n^3+22176n^2+26352n+12096)=0.
\]
Hence Lemma~\ref{lem:factor-criterion} applies with $c=\tfrac1{144}$, and \eqref{eq:BfromA} gives
\[
B_0(n)=(n+1)^3(3n+4)(3n+7),
\]
\[
B_1(n)=(2n+3)(3n+1)(3n+7)(5n^2+15n+12),
\]
\[
B_2(n)=16(n+2)^3(3n+1)(3n+4),
\]
which is \eqref{eq:L2secondkernel}.

\smallskip
\noindent\textbf{(iii) Printed order-$3$ Catalan recurrence.}
Appendix~B.6 of \cite[p.~21]{Raz} prints the common order-$3$ recurrence for two formulas for Catalan's constant:
\[
\begin{aligned}
0={}&\left(-\frac32-\frac{5n}{4}-\frac{n^2}{4}\right)f(n)
+\left(\frac{21}{2}+\frac{29n}{4}+\frac{5n^2}{4}\right)f(n+1)\\
&+\left(-\frac{85}{4}-13n-2n^2\right)f(n+2)
+\left(\frac{49}{4}+7n+n^2\right)f(n+3).
\end{aligned}
\]
Thus
\[
A_0(n)=-\frac{(n+2)(n+3)}{4},\qquad
A_1(n)=\frac{(n+3)(5n+14)}{4},
\]
\[
A_2(n)=-\frac{8n^2+52n+85}{4},\qquad
A_3(n)=\frac{(2n+7)^2}{4}.
\]
A direct expansion gives
\[
4\bigl(A_0+A_1+A_2+A_3\bigr)=-(n^2+5n+6)+(5n^2+29n+42)-(8n^2+52n+85)+(4n^2+28n+49)=0.
\]
Lemma~\ref{lem:factor-criterion} now applies with $c=\tfrac14$, and \eqref{eq:BfromA} yields
\[
B_0(n)=(n+1)(n+2),\qquad B_1(n)=-4(n+2)^2,\qquad B_2(n)=(2n+5)^2,
\]
which is \eqref{eq:L2catalan}, equivalently \eqref{eq:catalankernelrec}.
\end{proof}

\begin{remark}\label{rem:fourth-order3}
Raz et al.
report four order-$3$ canonical forms for $\pi$ \cite[\S3.3]{Raz}. The public PDF prints only two explicit order-$3$ $\pi$-recurrences in Appendix~B.6, together with one explicit order-$3$ Catalan recurrence. The fourth order-$3$ $\pi$-canonical form is not explicit in the public materials used here, so Theorem~\ref{thm:three-printed-factor} intentionally concerns only the three printed recurrences.
\end{remark}

\section{Ap\'ery-like kernels and Sym$^2$-to-square functoriality}\label{sec:aperylike-clausen}

\subsection{The first kernel and the sequence A036917}

We now identify the first kernel from Theorem~\ref{thm:three-printed-factor} with the classical sequence $A036917$.

\begin{theorem}\label{thm:A036917kernel}
Let
\[
F(z)=\Ftwoone\!\left(\tfrac12,\tfrac12;1;z\right),
\qquad
F(16x)^2=\sum_{n\ge0} A_n x^n.
\]
Then:
\begin{enumerate}[label=\textup{(\alph*)}]
\item
\[
A_n=\sum_{k=0}^n \binom{2k}{k}^2\binom{2n-2k}{n-k}^2,
\]
so $(A_n)_{n\ge0}$ is the classical sequence \textup{A036917}, one of the six sporadic Almkvist--Zudilin sequences \cite{Delaygue,OEISA036917}. The connection between this sequence and series for $1/\pi$ is well established \cite{ChanVerrill,Cooper}. Its standard recurrence is
\begin{equation}\label{eq:A036917rec}
n^3A_n=8(2n-1)(2n^2-2n+1)A_{n-1}-256(n-1)^3A_{n-2}
\qquad (n\ge1).
\end{equation}

\item The generating function also admits the Clausen-type pullback representation
\begin{equation}\label{eq:A036917-3F2}
\sum_{n\ge0} A_n x^n
=\FthreeTwo\!\left(\tfrac12,\tfrac12,\tfrac12;1,1;64x(1-16x)\right).
\end{equation}

\item If
\[
c_n:=\frac{nA_n}{32^n},
\]
then $c_n$ is annihilated by the operator $L^{(1)}_2$ from \eqref{eq:L2firstkernel}. Equivalently,
\begin{equation}\label{eq:c-from-A036917}
(n+1)^4 c_n-n(2n+3)(2n^2+6n+5)c_{n+1}+4n(n+1)(n+2)^2c_{n+2}=0.
\end{equation}
In particular, the first printed order-$3$ $\pi$-recurrence in Appendix~B.6 of \cite{Raz} is the summation lift of the $A036917$ kernel.
\end{enumerate}
\end{theorem}

\begin{proof}
For part \textup{(a)}, recall the well-known binomial expansion
\[
\Ftwoone\!\left(\tfrac12,\tfrac12;1;z\right)=\sum_{m\ge0}\frac{\left(\frac12\right)_m^2}{(m!)^2}z^m
=\sum_{m\ge0}\binom{2m}{m}^2\left(\frac{z}{16}\right)^m.
\]
Substituting $z=16x$ gives
\[
F(16x)=\sum_{m\ge0}\binom{2m}{m}^2 x^m.
\]
Squaring and taking the Cauchy product yields
\[
F(16x)^2=\sum_{n\ge0}\Bigl(\sum_{k=0}^n\binom{2k}{k}^2\binom{2n-2k}{n-k}^2\Bigr)x^n,
\]
which proves the convolution formula. The identification with A036917 and the recurrence \eqref{eq:A036917rec} are standard and recorded in \cite{Delaygue,OEISA036917}.

For part \textup{(b)}, we use the classical Clausen-type pullback identity
\begin{equation}\label{eq:clausen-pullback}
\FthreeTwo\!\left(\tfrac12,\tfrac12,\tfrac12;1,1;4u(1-u)\right)
=\Ftwoone\!\left(\tfrac12,\tfrac12;1;u\right)^2,
\end{equation}
see, for example, \cite[Eq.~(15) with $r=\tfrac12$]{Zudilin}. Setting $u=16x$ gives
\[
4u(1-u)=64x(1-16x),
\]
and therefore
\[
\FthreeTwo\!\left(\tfrac12,\tfrac12,\tfrac12;1,1;64x(1-16x)\right)
=\Ftwoone\!\left(\tfrac12,\tfrac12;1;16x\right)^2.
\]
The elementary identity
\[
(1-32x)^2=1-64x(1-16x)
\]
shows explicitly that the pullback is quadratic and regular at the origin.

For part \textup{(c)}, shift \eqref{eq:A036917rec} by one unit:
\[
(n+1)^3A_{n+1}=8(2n+1)(2n^2+2n+1)A_n-256n^3A_{n-1}.
\]
Now substitute
\[
A_n=\frac{32^n}{n}c_n \qquad (n\ge1).
\]
A direct substitution into \eqref{eq:A036917rec}, followed by division by $32^n$, clearing of the resulting denominators $n$ and $n-1$, and re-indexing, yields \eqref{eq:c-from-A036917}. This is precisely the kernel operator \eqref{eq:L2firstkernel} extracted in Theorem~\ref{thm:three-printed-factor}(i).
\end{proof}

\begin{remark}\label{rem:A036917-not-trajectory}
Theorem~\ref{thm:A036917kernel} does \emph{not} say that the current rank-$2$ $\pi$-CMF of \cite{Raz} already contains the $A036917$ kernel as a trajectory. What is proved here is a kernel-level identification. The intrinsic CMF origin of this kernel is supplied later by Theorem~\ref{thm:theta-bridge}, and it uses the differential component $M_\theta$ of the ambient Gauss CMF rather than a contiguous trajectory.
\end{remark}

\subsection{Identification of the remaining two kernels}

We now identify the second $\pi$-kernel and the Catalan kernel from Theorem~\ref{thm:three-printed-factor}.

\begin{theorem}[second $\pi$-kernel = Domb numbers]\label{thm:domb}
The second printed $\pi$-kernel \eqref{eq:L2secondkernel} is, after rescaling, the Domb number recurrence \textup{(OEIS A002895)} \cite{OEISA002895}. Specifically, if $(c_n)$ is annihilated by $L_2^{(2)}$, then the substitution
\begin{equation}\label{eq:domb-rescaling}
u_n=\frac{(-32)^n}{3n+1}\,c_n
\end{equation}
transforms the kernel recurrence into the Domb recurrence
\begin{equation}\label{eq:domb-rec}
(n+1)^3u_{n+1}=2(2n+1)(5n^2+5n+2)\,u_n-64n^3\,u_{n-1}.
\end{equation}
In the Almkvist--Zudilin/Delaygue classification this is case~$(\alpha)$ \cite{Delaygue,Gorodetsky}.
\end{theorem}

\begin{proof}
Write the kernel recurrence from \eqref{eq:L2secondkernel} as
\[
\begin{aligned}
0={}&(n+1)^3(3n+4)(3n+7)\,c_n \\
&+(2n+3)(3n+1)(3n+7)(5n^2+15n+12)\,c_{n+1} \\
&+16(n+2)^3(3n+1)(3n+4)\,c_{n+2}.
\end{aligned}
\]
Substituting $c_n=(3n+1)/(-32)^n\cdot u_n$, $c_{n+1}=(3n+4)/(-32)^{n+1}\cdot u_{n+1}$, $c_{n+2}=(3n+7)/(-32)^{n+2}\cdot u_{n+2}$ and dividing by $(3n+1)(3n+4)(3n+7)/(-32)^n$ gives
\[
(n+1)^3\,u_n
-\frac{(2n+3)(5n^2+15n+12)}{32}\,u_{n+1}
+\frac{(n+2)^3}{64}\,u_{n+2}=0.
\]
Multiplying by $64$ yields
\[
64(n+1)^3\,u_n-2(2n+3)(5n^2+15n+12)\,u_{n+1}+(n+2)^3\,u_{n+2}=0.
\]
Shifting $n\mapsto n-1$:
\[
(n+1)^3u_{n+1}-2(2n+1)(5n^2+5n+2)\,u_n+64n^3\,u_{n-1}=0,
\]
which is \eqref{eq:domb-rec}. This is exactly the standard recurrence of the Domb numbers $A002895$ \cite{OEISA002895}.
\end{proof}

\begin{theorem}[Catalan kernel from the Gauss-square family]\label{thm:catalan-square}
The printed Catalan kernel \eqref{eq:catalankernelrec} arises from the Gauss-square coefficient sequence at $(a,b,c)=(\tfrac12,1,\tfrac32)$ via a hypergeometric twist.

Specifically, let $g_n=[z^n]{}_2F_1(\tfrac12,1;\tfrac32;z)^2$, which satisfies the coefficient recurrence \eqref{eq:general-coeff-rec} at $(a,b,c)=(\tfrac12,1,\tfrac32)$:
\begin{equation}\label{eq:catalan-gauss-rec}
n(n+1)(2n+1)\,g_n-4n^3\,g_{n-1}+n(n-1)(2n-1)\,g_{n-2}=0.
\end{equation}
If the Catalan summand sequence $(c_n)$ from Theorem~\textup{\ref{thm:three-printed-factor}(iii)} satisfies \eqref{eq:catalankernelrec}, then there exists a nonzero constant $\kappa$ such that
\begin{equation}\label{eq:catalan-twist}
c_n=\kappa\,\frac{(n+1)!}{(2n+1)!!}\,g_n.
\end{equation}
With the normalization $c_0=\tfrac12$ and $g_0=1$, one has $\kappa=\tfrac12$, so \eqref{eq:catalan-twist} becomes
\[
c_n=\frac{(n+1)!}{2(2n+1)!!}\,g_n.
\]
In particular, together with Theorem~\ref{thm:theta-bridge} specialized at $(a,b,c)=(\tfrac12,1,\tfrac32)$ and Theorem~\ref{thm:three-printed-factor}(iii), this identifies the printed order-$3$ Catalan formula with the chain
\[
\text{ambient Gauss CMF}\longrightarrow \Sym^2 \longrightarrow M_\theta \longrightarrow \text{Catalan kernel}\longrightarrow \text{summation}.
\]
\end{theorem}

\begin{proof}
Set $c_n=t_n\,g_n$ in \eqref{eq:catalankernelrec}:
\[
(2n+5)^2\,t_{n+2}\,g_{n+2}-4(n+2)^2\,t_{n+1}\,g_{n+1}+(n+1)(n+2)\,t_n\,g_n=0.
\]
From \eqref{eq:catalan-gauss-rec} shifted by $2$, we can express $g_{n+2}$ in terms of $g_{n+1}$ and $g_n$:
\[
(n+3)(2n+5)\,g_{n+2}=4(n+2)^2\,g_{n+1}-(n+1)(2n+3)\,g_n.
\]
Substituting and collecting the coefficients of $g_{n+1}$ and $g_n$ gives
\[
4(n+2)^2\left(\frac{(2n+5)\,t_{n+2}}{n+3}-t_{n+1}\right)g_{n+1}
+(n+1)\left((n+2)\,t_n-\frac{(2n+3)(2n+5)}{n+3}\,t_{n+2}\right)g_n=0.
\]
Hence both coefficients vanish identically provided
\[
\frac{t_{n+1}}{t_n}=\frac{n+2}{2n+3}\qquad(n\ge0).
\]
This first-order ratio recurrence has the general solution
\[
t_n=\kappa\prod_{k=0}^{n-1}\frac{k+2}{2k+3}
=\kappa\,\frac{(n+1)!}{(2n+1)!!}
\]
for an arbitrary nonzero constant \(\kappa\), proving \eqref{eq:catalan-twist}. At $n=0$: $t_0=\kappa\cdot 1!/1!!=\kappa$ and $g_0=1$, so $c_0=t_0 g_0=\kappa$. With the printed Catalan normalization $c_0=\tfrac12$, this gives $\kappa=\tfrac12$. Since the Catalan kernel recurrence \eqref{eq:catalankernelrec} has order~$2$, any solution is determined by two consecutive initial values. To confirm agreement at $n=1$: $t_1=\kappa\cdot 2/3=1/3$ and $g_1=2/3$, giving $c_1=t_1 g_1=2/9$, which matches the printed Catalan summand. Hence $c_n=t_n g_n$ for all $n\ge0$.
\end{proof}

\begin{remark}[summary of kernel identifications]\label{rem:kernel-table}
Table~\ref{tab:kernels} collects the three kernel identifications. The first $\pi$-kernel and the Catalan kernel are explicit rescalings of coefficient sequences of Gauss squares ${}_2F_1(a,b;c;z)^2$ at specific parameter values; the second $\pi$-kernel (Domb numbers) is related to the ${}_2F_1$-world via an algebraic pullback rather than a same-variable square \cite{Gorodetsky}. All three are thus connected to the $\Sym^2({}_2F_1)$ framework, though by different mechanisms.
\end{remark}

\begin{table}[ht]
\centering
\caption{Identification of the three printed order-$3$ kernels.}\label{tab:kernels}
\smallskip
\renewcommand{\arraystretch}{1.15}
\begin{tabular}{@{}lcccc@{}}
\toprule
Kernel & OEIS / type & AZ case & Rescaling & Relation to ${}_2F_1^2$ \\
\midrule
First $\pi$ & $A036917$ (sporadic) & $(\beta)$ & $c_n=nA_n/32^n$ & rescaled from $(\tfrac12,\tfrac12,1)$ \\
Second $\pi$ & $A002895$, Domb (sporadic) & $(\alpha)$ & $u_n=(-32)^n c_n/(3n\!+\!1)$ & pullback-related \\
Catalan & non-sporadic & --- & $c_n=\kappa\,\dfrac{(n+1)!}{(2n+1)!!}\,g_n$ & twist from $(\tfrac12,1,\tfrac32)$ \\
\bottomrule
\end{tabular}

\smallskip
\parbox{0.9\linewidth}{\footnotesize For the Catalan kernel, \(\kappa\neq 0\) is arbitrary; the normalization used in Theorem~\ref{thm:catalan-square} is \(\kappa=\tfrac12\).}
\end{table}

\subsection{General square-gauge for Gauss CMFs}

For the remainder of this section let
\[
f(a,b,c;z):=\Ftwoone(a,b;c;z),\qquad g(a,b,c;z):=f(a,b,c;z)^2,
\]
and let $\theta=z\,\dd/\dd z$.

\begin{proposition}[general square-gauge for the Gauss CMF]\label{prop:general-square-gauge}
Let
\[
B_f=(f,\theta f),\qquad
B_{\mathrm{sym}}=(f^2,f\theta f,(\theta f)^2),\qquad
B_g=(g,\theta g,\theta^2 g).
\]
Assume that we work away from the rank-drop locus, so that for generic parameters the three entries of $B_{\mathrm{sym}}$ and of $B_g$ form bases of the rank-$3$ D-finite module generated by $g$. Then:
\begin{enumerate}[label=\textup{(\alph*)}]
\item the basis-change matrix from $B_{\mathrm{sym}}$ to $B_g$ is
\begin{equation}\label{eq:Phi-general}
\Phi(a,b,c;z)=
\begin{pmatrix}
1 & 0 & \dfrac{2abz}{1-z}\\[2mm]
0 & 2 & \dfrac{2((a+b)z-c+1)}{1-z}\\[2mm]
0 & 0 & 2
\end{pmatrix},
\qquad \det \Phi = 4;
\end{equation}
that is,
\begin{equation}\label{eq:basis-relation-general}
B_{\mathrm{sym}}\,\Phi(a,b,c;z)=B_g.
\end{equation}

\item for every parameter shift $u\in\Z^3$ for which the corresponding contiguous matrices are defined, the CMF generator of $g$ in basis $B_g$ is related to the Gauss-CMF generator of $f$ by
\begin{equation}\label{eq:general-square-gauge}
M_u^{(g)}(a,b,c;z)=\Phi(a,b,c;z)^{-1}\,\Sym^2\!\bigl(M_u^{(f)}(a,b,c;z)\bigr)\,\sigma_u(\Phi)(a,b,c;z).
\end{equation}
\end{enumerate}
\end{proposition}

\begin{proof}
By \cite[Example~3.3]{Weinbaum}, the Gauss hypergeometric function satisfies the Euler-form differential equation
\begin{equation}\label{eq:gauss-ode-general}
(1-z)\theta^2 f=abz\,f+\bigl((a+b)z-c+1\bigr)\theta f.
\end{equation}
From $g=f^2$ we get immediately
\[
g=f^2,\qquad \theta g=2f\theta f.
\]
Differentiating once more and using \eqref{eq:gauss-ode-general},
\[
\theta^2 g=\theta(2f\theta f)=2(\theta f)^2+2f\theta^2 f
=2(\theta f)^2+\frac{2((a+b)z-c+1)}{1-z}f\theta f+\frac{2abz}{1-z}f^2.
\]
These three identities are exactly the column equations encoded by \eqref{eq:basis-relation-general}; thus \eqref{eq:Phi-general} is the desired basis-change matrix, and $\det\Phi=4$ is immediate from its triangular form.

For part \textup{(b)}, by definition of the Gauss CMF,
\[
B_f\,M_u^{(f)}=\sigma_u(B_f).
\]
Since $B_{\mathrm{sym}}$ consists of quadratic monomials in the components of $B_f$, the defining formula \eqref{eq:sym2formula} for the symmetric square gives
\[
B_{\mathrm{sym}}\,\Sym^2(M_u^{(f)})=\sigma_u(B_{\mathrm{sym}}).
\]
On the other hand, by definition of the CMF generated by $g$ in the basis $B_g$,
\[
B_g\,M_u^{(g)}=\sigma_u(B_g).
\]
Using $B_g=B_{\mathrm{sym}}\Phi$ from part \textup{(a)}, we obtain
\[
B_{\mathrm{sym}}\Phi\,M_u^{(g)}
=\sigma_u(B_g)
=\sigma_u(B_{\mathrm{sym}})\sigma_u(\Phi)
=B_{\mathrm{sym}}\Sym^2(M_u^{(f)})\,\sigma_u(\Phi).
\]
Since the entries of $B_{\mathrm{sym}}$ form a basis, they may be cancelled on the left, yielding
\[
\Phi\,M_u^{(g)}=\Sym^2(M_u^{(f)})\,\sigma_u(\Phi),
\]
which is equivalent to \eqref{eq:general-square-gauge}. This is an explicit specialization of the abstract basis-change principle in \cite[Proposition~3.2]{Weinbaum}. On the rank-drop locus the same rational identities remain valid, but the interpretation is then as a relation between spanning triples inside a smaller module rather than between two rank-$3$ bases.
\end{proof}

\subsection{Pure Clausen corollary}

The pure Clausen identity is the special case in which the square itself becomes a hypergeometric function with the same evaluation variable.

\begin{corollary}[pure Clausen functoriality]\label{cor:clausen-functoriality}
Let
\[
f(a,b;z):=\Ftwoone\!\left(a,b;a+b+\tfrac12;z\right),
\]
and
\[
h(a,b;z):=\FthreeTwo\!\left(2a,2b,a+b;2a+2b,a+b+\tfrac12;z\right).
\]
Then \(h(a,b;z)=f(a,b;z)^2\). Let
\[
u=(1,0,1),\qquad v=(0,1,1)
\]
be the parameter-preserving shifts on the Gauss side, and let
\[
U=(2,0,1;2,1),\qquad V=(0,2,1;2,1)
\]
be the induced shifts on the \(\FthreeTwo\) side. If \(M_u^{(2F_1)},M_v^{(2F_1)}\) denote the restricted Gauss-CMF generators and \(M_U^{(3F_2)},M_V^{(3F_2)}\) the corresponding \(\FthreeTwo\)-CMF generators in the basis \((h,\theta h,\theta^2 h)\), then
\begin{equation}\label{eq:clausen-gauge-u-new}
M_U^{(3F_2)}=\Phi^{-1}\,\Sym^2\!\bigl(M_u^{(2F_1)}\bigr)\,\sigma_u(\Phi),
\end{equation}
\begin{equation}\label{eq:clausen-gauge-v-new}
M_V^{(3F_2)}=\Phi^{-1}\,\Sym^2\!\bigl(M_v^{(2F_1)}\bigr)\,\sigma_v(\Phi),
\end{equation}
where \(\Phi\) is \eqref{eq:Phi-general} specialized to \(c=a+b+\tfrac12\).
\end{corollary}

\begin{proof}
The classical Clausen identity \(h=f^2\) is standard; see \cite{Clausen1828,Vidunas}. On the hyperplane \(c=a+b+\tfrac12\), the induced shifts on the \(\FthreeTwo\) parameters are exactly \(U\) and \(V\). Applying Proposition~\ref{prop:general-square-gauge} to the function \(g=f^2=h\) on this restricted sublattice gives \eqref{eq:clausen-gauge-u-new} and \eqref{eq:clausen-gauge-v-new}.
\end{proof}

\subsection{The differential component and the A036917 recurrence}

The shift generators of a CMF are not the whole story: the D-finite construction of \cite{Weinbaum} also records the differential component \(M_\theta\). The third-order ODE for \(f^2\) where \(f={}_2F_1(a,b;c;z)\) is classical, first derived by Chaundy \cite{Chaundy} and explicit in Vid\=unas \cite[Eq.~(31)]{Vidunas}. The general coefficient recurrence for \({}_2F_1(a,b;c;z)^k\) in the cases \(k=2\) and \(k=3\) has recently been obtained by Mao and Tian \cite{MaoTian}. We reformulate these classical results in the CMF framework and show that the first printed \(\pi\)-kernel arises from the differential component after passing to the square.

\begin{theorem}[CMF reformulation of the Chaundy--Vid\=unas ODE]\label{thm:theta-bridge}
Let \(f(a,b,c;z)={}_2F_1(a,b;c;z)\) and \(g=f^2\). Work away from the rank-drop locus, as in Proposition~\textup{\ref{prop:general-square-gauge}}.
\begin{enumerate}[label=\textup{(\alph*)}]
\item In the square basis
\[
B_{\mathrm{sym}}=(f^2,f\theta f,(\theta f)^2),
\]
the differential component \(\theta\) acts by
\begin{equation}\label{eq:Mtheta-sym-general}
M_{\theta,\mathrm{sym}}=
\begin{pmatrix}
0 & \dfrac{abz}{1-z} & 0\\[2mm]
2 & \dfrac{(a+b)z-c+1}{1-z} & \dfrac{2abz}{1-z}\\[2mm]
0 & 1 & \dfrac{2((a+b)z-c+1)}{1-z}
\end{pmatrix}.
\end{equation}

\item In the basis \(B_g=(g,\theta g,\theta^2 g)\), the differential component is the companion matrix
\begin{equation}\label{eq:Mtheta-g-general}
M_\theta^{(g)}=
\begin{pmatrix}
0 & 0 & \dfrac{2abz(2c-1-2(a+b)z)}{(1-z)^2}\\[2mm]
1 & 0 & \dfrac{-2(a^2+4ab+b^2)z^2 +(4ab+4ac+4bc-3a-3b-c+1)z -2(c-1)^2}{(1-z)^2}\\[2mm]
0 & 1 & \dfrac{3((a+b)z-c+1)}{1-z}
\end{pmatrix}.
\end{equation}
Equivalently, \(g\) is annihilated by the third-order Euler differential operator first derived by Chaundy \cite{Chaundy} and recorded explicitly by Vid\=unas \cite[Eq.~(31)]{Vidunas}:
\begin{equation}\label{eq:square-operator-general}
\mathcal L_{a,b,c}
=
2\theta(\theta+c-1)(\theta+2c-2)
-
z\,Q_{a,b,c}(\theta)
+
2z^2(\theta+a+b)(\theta+2a)(\theta+2b),
\end{equation}
where
\begin{equation}\label{eq:Qabc}
Q_{a,b,c}(T)=
4T^3+6(a+b+c-1)T^2
+2(4ab+4ac+4bc-3a-3b-c+1)T
+4ab(2c-1).
\end{equation}

\item If \(g(z)=\sum_{n\ge0}g_n z^n\), then the coefficients satisfy the recurrence (a special case of Mao--Tian \cite[Theorem~2.1]{MaoTian})
\begin{equation}\label{eq:general-coeff-rec}
2n(n+c-1)(n+2c-2)\,g_n
-
Q_{a,b,c}(n-1)\,g_{n-1}
+
2(n+a+b-2)(n+2a-2)(n+2b-2)\,g_{n-2}=0
\end{equation}
for \(n\ge2\).

\item In the \(\pi\)-case \((a,b,c)=(\tfrac12,\tfrac12,1)\),
\begin{equation}\label{eq:Phi-pi}
\Phi_\pi(z)=
\begin{pmatrix}
1 & 0 & \dfrac{z}{2(1-z)}\\[2mm]
0 & 2 & \dfrac{2z}{1-z}\\[2mm]
0 & 0 & 2
\end{pmatrix},
\end{equation}
and
\begin{equation}\label{eq:Mtheta-sym-pi}
M_{\theta,\mathrm{sym}}^{(\pi)}=
\begin{pmatrix}
0 & \dfrac{z}{4(1-z)} & 0\\[2mm]
2 & \dfrac{z}{1-z} & \dfrac{z}{2(1-z)}\\[2mm]
0 & 1 & \dfrac{2z}{1-z}
\end{pmatrix},
\end{equation}
\begin{equation}\label{eq:Mtheta-g-pi}
M_\theta^{(\pi,\mathrm{sq})}=
\begin{pmatrix}
0 & 0 & \dfrac{z(1-2z)}{2(1-z)^2}\\[2mm]
1 & 0 & \dfrac{z(2-3z)}{(1-z)^2}\\[2mm]
0 & 1 & \dfrac{3z}{1-z}
\end{pmatrix}.
\end{equation}
Therefore \(g(z)= {}_2F_1(\tfrac12,\tfrac12;1;z)^2\) satisfies
\begin{equation}\label{eq:square-ode-pi-left}
2\theta^3 g-z(2\theta+1)(2\theta^2+2\theta+1)g+2z^2(\theta+1)^3g=0,
\end{equation}
equivalently
\begin{equation}\label{eq:square-ode-pi-right}
2\theta^3 g-(2\theta-1)(2\theta^2-2\theta+1)z\,g+2(\theta-1)^3z^2 g=0.
\end{equation}
If \(g(z)=\sum_{n\ge0}g_n z^n\), then
\begin{equation}\label{eq:gk-rec}
2n^3g_n-(2n-1)(2n^2-2n+1)g_{n-1}+2(n-1)^3g_{n-2}=0.
\end{equation}
Finally, if \(A_n:=16^n g_n\), then \((A_n)\) satisfies \eqref{eq:A036917rec}; hence \(A_n\) is exactly the sequence of Theorem~\ref{thm:A036917kernel}.
\end{enumerate}
\end{theorem}

\begin{proof}
We proceed step by step.

For part \textup{(a)}, let
\[
b_0=f^2,\qquad b_1=f\theta f,\qquad b_2=(\theta f)^2.
\]
The Gauss equation \eqref{eq:gauss-ode-general} yields
\[
\theta b_0=\theta(f^2)=2f\theta f=2b_1.
\]
Next,
\[
\theta b_1=\theta(f\theta f)=(\theta f)^2+f\theta^2 f.
\]
Using \eqref{eq:gauss-ode-general},
\[
f\theta^2 f
=
\frac{abz}{1-z}f^2+\frac{(a+b)z-c+1}{1-z}f\theta f
=
\frac{abz}{1-z}b_0+\frac{(a+b)z-c+1}{1-z}b_1,
\]
hence
\[
\theta b_1=\frac{abz}{1-z}b_0+\frac{(a+b)z-c+1}{1-z}b_1+b_2.
\]
Finally,
\[
\theta b_2=\theta\bigl((\theta f)^2\bigr)=2\theta f\,\theta^2 f
=
\frac{2abz}{1-z}f\theta f+\frac{2((a+b)z-c+1)}{1-z}(\theta f)^2
=
\frac{2abz}{1-z}b_1+\frac{2((a+b)z-c+1)}{1-z}b_2.
\]
These three identities are exactly the column equations encoded by \eqref{eq:Mtheta-sym-general}.

For part \textup{(b)}, we use the general gauge matrix \(\Phi\) from Proposition~\ref{prop:general-square-gauge}. Since
\[
B_g=B_{\mathrm{sym}}\Phi,
\]
differentiating gives
\[
\theta(B_g)=\theta(B_{\mathrm{sym}})\Phi+B_{\mathrm{sym}}\theta(\Phi)
=
B_{\mathrm{sym}}\bigl(M_{\theta,\mathrm{sym}}\Phi+\theta(\Phi)\bigr).
\]
By definition of \(M_\theta^{(g)}\),
\[
\theta(B_g)=B_gM_\theta^{(g)}=B_{\mathrm{sym}}\Phi\,M_\theta^{(g)}.
\]
Cancelling the basis \(B_{\mathrm{sym}}\) yields the differential gauge formula
\[
M_\theta^{(g)}=\Phi^{-1}\bigl(M_{\theta,\mathrm{sym}}\Phi+\theta(\Phi)\bigr).
\]
Substituting \eqref{eq:Mtheta-sym-general} and \eqref{eq:Phi-general} and simplifying gives \eqref{eq:Mtheta-g-general}. Because this is a companion matrix, the corresponding differential equation is
\[
\theta^3 g-m_2\theta^2 g-m_1\theta g-m_0 g=0,
\]
where \(m_0,m_1,m_2\) are the three entries in the last column of \eqref{eq:Mtheta-g-general}. Multiplying by \((1-z)^2\) and collecting the coefficients of \(1\), \(z\), and \(z^2\) gives
\[
2(1-z)^2\bigl(\theta^3-m_2\theta^2-m_1\theta-m_0\bigr)
=
2\theta(\theta+c-1)(\theta+2c-2)-zQ_{a,b,c}(\theta)+2z^2(\theta+a+b)(\theta+2a)(\theta+2b),
\]
with \(Q_{a,b,c}\) as in \eqref{eq:Qabc}. This is \eqref{eq:square-operator-general}.

For part \textup{(c)}, write
\[
g(z)=\sum_{n\ge0}g_n z^n.
\]
For any polynomial \(P\) in \(\theta\),
\[
[z^n]\bigl(P(\theta)g\bigr)=P(n)g_n,
\]
because \(\theta^k z^n=n^k z^n\). Likewise,
\[
[z^n]\bigl(zP(\theta)g\bigr)=P(n-1)g_{n-1},
\qquad
[z^n]\bigl(z^2P(\theta)g\bigr)=P(n-2)g_{n-2}.
\]
Extracting the coefficient of \(z^n\) from \eqref{eq:square-operator-general} therefore yields \eqref{eq:general-coeff-rec}.

For part \textup{(d)}, substitute \(a=b=\tfrac12\), \(c=1\) into \eqref{eq:Phi-general} and \eqref{eq:Mtheta-sym-general}; this gives \eqref{eq:Phi-pi} and \eqref{eq:Mtheta-sym-pi}. Substituting the same values into \eqref{eq:Mtheta-g-general} gives \eqref{eq:Mtheta-g-pi}. Reading off the corresponding companion differential equation yields
\[
\theta^3 g-\frac{3z}{1-z}\theta^2g-\frac{z(2-3z)}{(1-z)^2}\theta g-\frac{z(1-2z)}{2(1-z)^2}g=0.
\]
Multiplying by \(2(1-z)^2\) gives
\[
2(1-z)^2\theta^3g-6z(1-z)\theta^2 g-2z(2-3z)\theta g-z(1-2z)g=0.
\]
Expanding the right-hand side of \eqref{eq:square-ode-pi-left},
\[
2\theta^3-z(2\theta+1)(2\theta^2+2\theta+1)+2z^2(\theta+1)^3,
\]
produces exactly the same operator:
\[
2\theta^3-4z\theta^3-6z\theta^2-4z\theta-z+2z^2\theta^3+6z^2\theta^2+6z^2\theta+2z^2.
\]
Hence \eqref{eq:square-ode-pi-left} holds. The right-shifted form \eqref{eq:square-ode-pi-right} follows from the Euler commutation rule
\[
\theta z = z(\theta+1),
\]
equivalently \(zP(\theta)=P(\theta-1)z\), applied twice.

Finally, coefficient extraction from \eqref{eq:square-ode-pi-left} gives
\[
2n^3g_n-(2(n-1)+1)\bigl(2(n-1)^2+2(n-1)+1\bigr)g_{n-1}+2(n-1)^3 g_{n-2}=0,
\]
which simplifies to \eqref{eq:gk-rec}. If \(A_n=16^n g_n\), then multiplying \eqref{eq:gk-rec} by \(16^n\) and dividing by \(2\) yields
\[
n^3A_n=8(2n-1)(2n^2-2n+1)A_{n-1}-256(n-1)^3A_{n-2},
\]
namely \eqref{eq:A036917rec}. This is exactly the standard recurrence of \(A036917\).
\end{proof}

\begin{corollary}[the first printed order-$3$ \(\pi\)-formula from the ambient Gauss CMF]\label{cor:first-pi-chain}
Let \(f(z)=\Ftwoone(\tfrac12,\tfrac12;1;z)\). The first printed order-$3$ \(\pi\)-formula in Appendix~B.6 of \cite{Raz} is obtained from the \emph{ambient} Gauss $\Ftwoone$-CMF (whose evaluation-$z=\tfrac12$ slice is the published rank-$2$ $\pi$-CMF of \cite{Raz}) by the following operations:
\[
\text{square } f\ \longmapsto\ \text{differential component }M_\theta\ \longmapsto\ \text{coefficient extraction}\ \longmapsto\ \text{summation}.
\]
\end{corollary}

\begin{proof}
Theorem~\ref{thm:theta-bridge}(d) identifies the coefficient recurrence of \(g(z)=f(z)^2\) with the A036917 recurrence. Theorem~\ref{thm:A036917kernel}(c) transforms that coefficient recurrence into the kernel \(L_2^{(1)}\) of the first printed \(\pi\)-formula. Theorem~\ref{thm:three-printed-factor}(i) then lifts \(L_2^{(1)}\) to the printed order-$3$ formula recurrence by one summation.
\end{proof}

\begin{remark}[nonlinear pullback versus differential extraction]\label{rem:pullback-obstruction}
The pure Clausen corollary is a \emph{parameter} statement: it compares CMFs attached to the same evaluation variable \(z\). The \(A036917\) identity from Theorem~\ref{thm:A036917kernel}(b) involves instead the nonlinear pullback
\[
z\mapsto 64x(1-16x)
\]
in the evaluation variable. Theorem~\ref{thm:theta-bridge} bypasses this issue by extracting coefficients from the differential component \(M_\theta\) of the ambient square CMF. Theorem~\ref{thm:pullback-cmf} below provides a rational pullback--twist transport theorem; what remains open is a fully general theory encompassing algebraic pullbacks in the mixed shift-differential setting.
\end{remark}

\subsection{Inverse classification via the accessory parameter}

For a generic Riemann scheme of $\Sym^2$-type, the local exponents alone do not determine the differential operator: one accessory parameter remains free. The following result shows that the $\Sym^2(\text{Gauss})$ locus is cut out by a single explicit equation.

\begin{theorem}[inverse classification]\label{thm:inverse}
Let $\alpha,\beta,\gamma_1,\gamma_2$ satisfy $\alpha+\beta+\gamma_1+\gamma_2=1$. The theorem is first proved on the generic locus where $\alpha,2\alpha,\beta,2\beta,\gamma_1-\gamma_2\notin\mathbb Z\setminus\{0\}$ and the operator has genuine order~$3$; since all formulas are rational in the parameters, they extend by specialization to all non-degenerate points, including the $A036917$ case $(\alpha,\beta,\gamma_1,\gamma_2)=(0,0,\tfrac12,\tfrac12)$ and the Catalan case $(\alpha,\beta,\gamma_1,\gamma_2)=(-\tfrac12,0,\tfrac12,1)$ of Remark~\textup{\ref{rem:inverse-check}}. Consider the $\Sym^2$-type Riemann scheme
\[
z=0:\ \{0,\alpha,2\alpha\},\qquad
z=1:\ \{0,\beta,2\beta\},\qquad
z=\infty:\ \{\gamma_1+\gamma_2,\,2\gamma_1,\,2\gamma_2\}.
\]
\begin{enumerate}[label=\textup{(\alph*)}]
\item There exists a one-parameter family $L_\lambda$ of monic order-$3$ Fuchsian operators on $\mathbb P^1\setminus\{0,1,\infty\}$ with this Riemann scheme, parametrized by an accessory parameter $\lambda$. In Euler form,
\begin{equation}\label{eq:L-lambda}
\widetilde L_\lambda
=
\theta(\theta-\alpha)(\theta-2\alpha)
+z\bigl(P_1(\theta)-\lambda\bigr)
+z^2(\theta+\gamma_1+\gamma_2)(\theta+2\gamma_1)(\theta+2\gamma_2),
\end{equation}
where $P_1(\theta)=-2\theta^3+3(\alpha-\gamma_1-\gamma_2)\theta^2+(4\alpha(\gamma_1+\gamma_2)-4\gamma_1\gamma_2-\alpha-\gamma_1-\gamma_2)\theta$.

\item The operator $L_\lambda$ is the symmetric square of a Gauss hypergeometric equation if and only if
\begin{equation}\label{eq:lambda0}
\boxed{\lambda=\lambda_0:=2\gamma_1\gamma_2(1-2\alpha).}
\end{equation}

\item At $\lambda=\lambda_0$, the recovered Gauss parameters are
\begin{equation}\label{eq:recovered-params}
a=\gamma_1,\qquad b=\gamma_2,\qquad c=1-\alpha,
\end{equation}
and $L_{\lambda_0}=\Sym^2\bigl(\textup{Gauss}(a,b;c)\bigr)$.
\end{enumerate}
\end{theorem}

\begin{proof}
We work on the Zariski-open set where the operator has genuine order \(3\) and all exponent differences are nonintegral; every identity obtained below is rational in the parameters and therefore extends by specialization to all non-degenerate points.

\smallskip
\noindent\textbf{Part \textup{(a)}.}
A monic third-order Fuchsian operator with singularities only at \(0,1,\infty\) and with one zero exponent at both \(0\) and \(1\) can be written in the form
\[
L=D^3+A(z)D^2+B(z)D+C(z),
\]
with
\[
A(z)=\frac{u_0}{z}+\frac{u_1}{z-1},\qquad
B(z)=\frac{v_0}{z^2}+\frac{v_1}{(z-1)^2}+\frac{v}{z(z-1)},
\]
\[
C(z)=\frac{s}{z^2(z-1)}+\frac{t}{z(z-1)^2}.
\]
The indicial polynomial at \(z=0\) is obtained by substituting \(y=z^r\):
\[
I_0(r)=r\bigl((r-1)(r-2)+u_0(r-1)+v_0\bigr).
\]
Requiring \(I_0(r)=r(r-\alpha)(r-2\alpha)\) gives
\[
u_0=3-3\alpha,\qquad v_0=(1-\alpha)(1-2\alpha).
\]
Likewise, substituting \(y=(z-1)^r\) gives
\[
I_1(r)=r\bigl((r-1)(r-2)+u_1(r-1)+v_1\bigr),
\]
and the condition \(I_1(r)=r(r-\beta)(r-2\beta)\) yields
\[
u_1=3-3\beta,\qquad v_1=(1-\beta)(1-2\beta).
\]

At \(z=\infty\), substituting \(y=z^{-r}\) shows that the coefficient of \(z^{-r-3}\) in \(L(y)\) is
\[
I_\infty(r)=-r(r+1)(r+2)+(u_0+u_1)r(r+1)-(v_0+v_1+v)r+(s+t).
\]
We require
\[
I_\infty(r)=-(r-\gamma_1-\gamma_2)(r-2\gamma_1)(r-2\gamma_2).
\]
Matching the coefficients of \(r\) and the constant term, and using the Fuchs relation
\[
\alpha+\beta+\gamma_1+\gamma_2=1,
\]
gives
\[
v=4(\alpha\beta+\gamma_1\gamma_2+\gamma_1+\gamma_2),
\qquad
s+t=4\gamma_1\gamma_2(\gamma_1+\gamma_2).
\]
Hence exactly one free parameter remains. We write
\[
s=\lambda,\qquad t=4\gamma_1\gamma_2(\gamma_1+\gamma_2)-\lambda.
\]

Now multiply the operator \(L\) by \(z^3(z-1)^2\) and use \(zD=\theta\). Evaluating on a monomial \(z^n\) gives
\[
z^3(z-1)^2L(z^n)=z^n\bigl[P_0(n)+z\bigl(P_1(n)-\lambda\bigr)+z^2P_2(n)\bigr],
\]
where
\[
P_0(n)=n(n-\alpha)(n-2\alpha),
\]
\[
P_1(n)=-2n^3+3(\alpha-\gamma_1-\gamma_2)n^2
+\bigl(4\alpha(\gamma_1+\gamma_2)-4\gamma_1\gamma_2-\alpha-\gamma_1-\gamma_2\bigr)n,
\]
\[
P_2(n)=(n+\gamma_1+\gamma_2)(n+2\gamma_1)(n+2\gamma_2).
\]
Since this identity holds for every \(n\), the corresponding Euler-form operator is precisely
\[
\widetilde L_\lambda
=
\theta(\theta-\alpha)(\theta-2\alpha)
+z\bigl(P_1(\theta)-\lambda\bigr)
+z^2(\theta+\gamma_1+\gamma_2)(\theta+2\gamma_1)(\theta+2\gamma_2),
\]
which is \eqref{eq:L-lambda}.

\smallskip
\noindent\textbf{Part \textup{(b)}.}
Assume now that \(L_\lambda\) is the symmetric square of a second-order operator
\[
M=D^2+p(z)D+q(z).
\]
A standard calculation gives
\[
\Sym^2(M)=D^3+3p\,D^2+\bigl(2p^2+p'+4q\bigr)D+\bigl(4pq+2q'\bigr).
\]
Because the local exponent sets of \(M\) must be \(\{0,\alpha\}\) at \(0\), \(\{0,\beta\}\) at \(1\), and \(\{\gamma_1,\gamma_2\}\) at \(\infty\), the coefficients of \(M\) are forced to be
\[
p(z)=\frac{1-\alpha}{z}+\frac{1-\beta}{z-1},
\qquad
q(z)=\frac{\gamma_1\gamma_2}{z(z-1)}.
\]
Indeed, these are exactly the partial fractions of the monic Gauss operator with parameters \(a=\gamma_1\), \(b=\gamma_2\), \(c=1-\alpha\).

Substituting \(p\) and \(q\) into the symmetric-square formulas gives
\[
3p(z)=\frac{3-3\alpha}{z}+\frac{3-3\beta}{z-1}=A(z),
\]
and
\[
2p(z)^2+p'(z)+4q(z)
=
\frac{(1-\alpha)(1-2\alpha)}{z^2}
+\frac{(1-\beta)(1-2\beta)}{(z-1)^2}
+\frac{4(\alpha\beta+\gamma_1\gamma_2+\gamma_1+\gamma_2)}{z(z-1)}
=
B(z).
\]
So the \(D^2\)- and \(D\)-coefficients already agree with the family from part \textup{(a)}.

For the constant term we compute
\[
4p(z)q(z)+2q'(z)
=
\frac{2\gamma_1\gamma_2(2\alpha-1)}{z^2}
+\frac{4\gamma_1\gamma_2(\alpha-\beta)}{z}
-\frac{2\gamma_1\gamma_2(2\beta-1)}{(z-1)^2}
-\frac{4\gamma_1\gamma_2(\alpha-\beta)}{z-1}.
\]
On the other hand, the partial fraction decomposition of
\[
C_\lambda(z)=\frac{\lambda}{z^2(z-1)}+\frac{4\gamma_1\gamma_2(\gamma_1+\gamma_2)-\lambda}{z(z-1)^2}
\]
is
\[
C_\lambda(z)
=
-\frac{\lambda}{z^2}
+\frac{4\gamma_1\gamma_2(\gamma_1+\gamma_2)-2\lambda}{z}
+\frac{4\gamma_1\gamma_2(\gamma_1+\gamma_2)-\lambda}{(z-1)^2}
-\frac{4\gamma_1\gamma_2(\gamma_1+\gamma_2)-2\lambda}{z-1}.
\]
Using \(\beta=1-\alpha-\gamma_1-\gamma_2\), the difference simplifies to
\[
C_\lambda-(4pq+2q')
=
-\frac{\lambda+4\alpha\gamma_1\gamma_2-2\gamma_1\gamma_2}{z^2(z-1)^2}.
\]
Therefore \(C_\lambda=4pq+2q'\) if and only if
\[
\lambda=2\gamma_1\gamma_2(1-2\alpha),
\]
which is exactly \eqref{eq:lambda0}. This proves the necessity of \eqref{eq:lambda0}. Conversely, when \(\lambda=\lambda_0\), the coefficients \(A,B,C\) of \(L_\lambda\) agree term-by-term with those of \(\Sym^2(M)\), hence \(L_{\lambda_0}=\Sym^2(M)\).

\smallskip
\noindent\textbf{Part \textup{(c)}.}
With \(\lambda=\lambda_0\), the recovered second-order operator is
\[
M=D^2+\left(\frac{1-\alpha}{z}+\frac{1-\beta}{z-1}\right)D+\frac{\gamma_1\gamma_2}{z(z-1)}.
\]
Set
\[
a=\gamma_1,\qquad b=\gamma_2,\qquad c=1-\alpha.
\]
Then
\[
c-a-b=1-\alpha-\gamma_1-\gamma_2=\beta,
\qquad
ab=\gamma_1\gamma_2,
\]
so \(M\) is the monic Gauss hypergeometric operator for \({}_2F_1(a,b;c;z)\). Therefore
\[
L_{\lambda_0}=\Sym^2(\mathrm{Gauss}(a,b;c)),
\]
which proves \eqref{eq:recovered-params}.
\end{proof}

\begin{remark}[verification on the printed kernels]\label{rem:inverse-check}
For the $A036917$ case $(\alpha,\beta,\gamma_1,\gamma_2)=(0,0,\tfrac12,\tfrac12)$, formula \eqref{eq:lambda0} gives $\lambda_0=\tfrac12$. The Frobenius recurrence of $\widetilde L_{1/2}$ is
\[
2n^3g_n-(2n-1)(2n^2-2n+1)g_{n-1}+2(n-1)^3g_{n-2}=0,
\]
which is exactly the $A036917$ recurrence.

For the Catalan case $(\alpha,\beta,\gamma_1,\gamma_2)=(-\tfrac12,0,\tfrac12,1)$, formula \eqref{eq:lambda0} gives $\lambda_0=2$. The Frobenius recurrence of $\widetilde L_2$ reproduces the Catalan-kernel Gauss-square recurrence at $(a,b,c)=(\tfrac12,1,\tfrac32)$.
\end{remark}

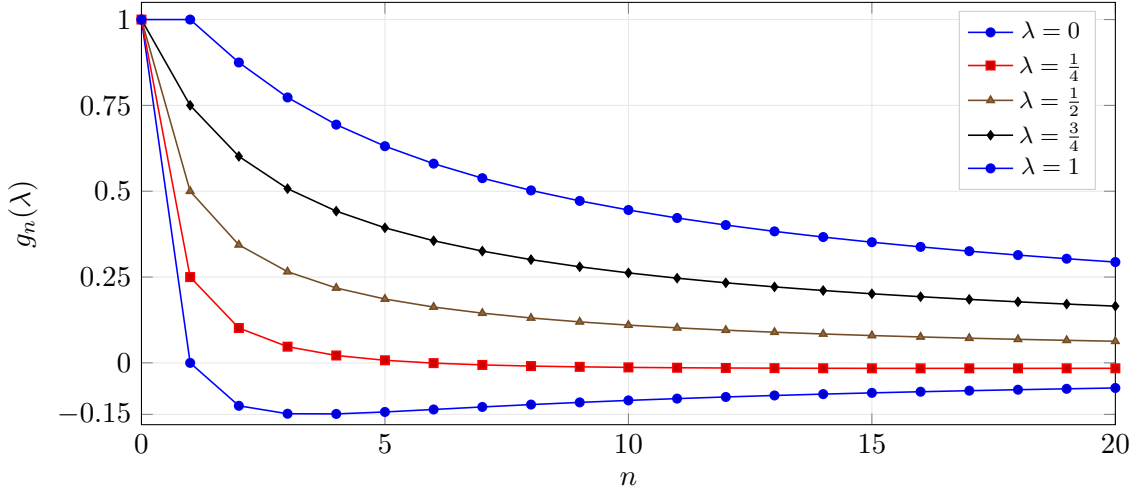
\begin{figure}[t]
\centering
\begin{tikzpicture}
\begin{axis}[
    width=.88\textwidth,
    height=7.2cm,
    xlabel={$n$},
    ylabel={$g_n(\lambda)$},
    xmin=0, xmax=20,
    ymin=-0.18, ymax=1.05,
    xtick={0,5,10,15,20},
    ytick={-0.15,0,0.25,0.5,0.75,1.0},
    grid=both,
    major grid style={gray!18},
    minor grid style={gray!10},
    legend style={at={(0.98,0.98)},anchor=north east,fill=white,draw=gray!40,font=\footnotesize},
    every axis plot/.append style={semithick, mark size=1.6pt},
]
\addplot+[mark=*] coordinates {(0,1.0000000000) (1,0.0000000000) (2,-0.1250000000) (3,-0.1481481481) (4,-0.1486545139) (5,-0.1429675926) (6,-0.1357052282) (7,-0.1283623497) (8,-0.1214375331) (9,-0.1150747842) (10,-0.1092855925) (11,-0.1040329003) (12,-0.0992645391) (13,-0.0949267298) (14,-0.0909693168) (15,-0.0873474854) (16,-0.0840219871) (17,-0.0809587641) (18,-0.0781283677) (19,-0.0755053452) (20,-0.0730676679)};
\addlegendentry{$\lambda=0$}
\addplot+[mark=square*] coordinates {(0,1.0000000000) (1,0.2500000000) (2,0.1015625000) (3,0.0472366898) (4,0.0215502138) (5,0.0075798300) (6,-0.0007066607) (7,-0.0059110690) (8,-0.0093081418) (9,-0.0115822913) (10,-0.0131272190) (11,-0.0141819223) (12,-0.0148978786) (13,-0.0153745397) (14,-0.0156791405) (15,-0.0158582686) (16,-0.0159448871) (17,-0.0159627403) (18,-0.0159291968) (19,-0.0158571311) (20,-0.0157561965)};
\addlegendentry{$\lambda=\tfrac14$}
\addplot+[mark=triangle*] coordinates {(0,1.0000000000) (1,0.5000000000) (2,0.3437500000) (3,0.2656250000) (4,0.2181396484) (5,0.1859741211) (6,0.1626243591) (7,0.1448383331) (8,0.1308014300) (9,0.1194177447) (10,0.1099845695) (11,0.1020296386) (12,0.0952232995) (13,0.0893281492) (14,0.0841686830) (15,0.0796122323) (16,0.0755565656) (17,0.0719215780) (18,0.0686435736) (19,0.0656712412) (20,0.0629627641)};
\addlegendentry{$\lambda=\tfrac12$}
\addplot+[mark=diamond*] coordinates {(0,1.0000000000) (1,0.7500000000) (2,0.6015625000) (3,0.5074508102) (4,0.4419784546) (5,0.3934293411) (6,0.3557712706) (7,0.3255738546) (8,0.3007337684) (9,0.2798844163) (10,0.2620963077) (11,0.2467134409) (12,0.2332585093) (13,0.2213752418) (14,0.2107918742) (15,0.2012971943) (16,0.1927243605) (17,0.1849396876) (18,0.1778346966) (19,0.1713203665) (20,0.1653229030)};
\addlegendentry{$\lambda=\tfrac34$}
\addplot+[mark=otimes*] coordinates {(0,1.0000000000) (1,1.0000000000) (2,0.8750000000) (3,0.7731481481) (4,0.6939380787) (5,0.6311765046) (6,0.5802455820) (7,0.5380223961) (8,0.5023803853) (9,0.4718347712) (10,0.4453205057) (11,0.4220541866) (12,0.4014465092) (13,0.3830453144) (14,0.3664975689) (15,0.3515233526) (16,0.3378976617) (17,0.3254374148) (18,0.3139920052) (19,0.3034363142) (20,0.2936654673)};
\addlegendentry{$\lambda=1$}
\end{axis}
\end{tikzpicture}
\caption{The one-parameter deformation in Theorem~\ref{thm:inverse} for the $A036917$ Riemann scheme \((\alpha,\beta,\gamma_1,\gamma_2)=(0,0,\tfrac12,\tfrac12)\). The coefficients \(g_n(\lambda)\) are generated by the Frobenius recurrence of \(\widetilde L_\lambda\) with \(g_0=1\) and \(g_1=\lambda\). The special value \(\lambda_0=\tfrac12\) is the unique \(\Sym^2(\mathrm{Gauss})\) point; after the rescaling \(A_n=16^n g_n(\tfrac12)\), one recovers \(A036917\).}
\label{fig:lambda-deformation}
\end{figure}

\subsection{Belyi pullback closure for the Domb kernel}

The Domb kernel $A002895$ is not a direct coefficient sequence of any ${}_2F_1(a,b;c;z)^2$, since its generating-function ODE has four singularities rather than three. However, the Domb generating function admits a classical representation as a \emph{pulled-back twisted Gauss square}. We now show that this representation closes the CMF chain for the second printed $\pi$-kernel.

\begin{theorem}[Domb closure via Belyi pullback]\label{thm:domb-pullback}
Let
\[
f(z)={}_2F_1\!\left(\tfrac16,\tfrac13;1;z\right),
\qquad
\phi(x)=\frac{108x^2}{(1-4x)^3},
\qquad
y(x)=\frac{1}{\sqrt{1-4x}}\,f\bigl(\phi(x)\bigr).
\]
Then:
\begin{enumerate}[label=\textup{(\alph*)}]
\item The map $\phi\colon\mathbb P^1\to\mathbb P^1$ is a Belyi map of degree $3$, ramified only over $\{0,1,\infty\}$, with ramification passport $[2{+}1,\, 2{+}1,\, 3]$.

\item The function $y(x)$ satisfies the Heun equation
\begin{equation}\label{eq:heun-domb}
x(1{-}4x)(1{-}16x)\,y''+(1{-}30x+128x^2)\,y'-2(1{-}8x)\,y=0.
\end{equation}

\item $G(x):=y(x)^2=\sum_{n\ge0}D_n\,x^n$, where $(D_n)$ are the Domb numbers $A002895$. The function $G$ satisfies $\Sym^2$ of \eqref{eq:heun-domb}:
\begin{equation}\label{eq:domb-ode-theta}
\bigl[\theta_x^3-2x(2\theta_x{+}1)(5\theta_x^2{+}5\theta_x{+}2)+64x^2(\theta_x{+}1)^3\bigr]G=0.
\end{equation}

\item In particular, the Domb kernel arises from the rank-$2$ Gauss object ${}_2F_1(\tfrac16,\tfrac13;1;z)$ by three successive operations: Belyi pullback~$\phi$, algebraic twist by $1/\sqrt{1-4x}$, and $\Sym^2$.
\end{enumerate}
\end{theorem}

\begin{proof}
Part~(a): one computes $\phi'(x)=216x(1+2x)/(1-4x)^4$, so the finite critical points are $x=0$ and $x=-\tfrac12$. Their images are $\phi(0)=0$ and $\phi(-\tfrac12)=1$, while $\phi$ has a triple pole at $x=\tfrac14$ mapping to $\infty$. Thus $\phi$ is ramified only over $\{0,1,\infty\}$. The identity $1-\phi(x)=(1+2x)^2(1-16x)/(1-4x)^3$ confirms the ramification pattern: the double factor $(1+2x)^2$ accounts for the order-$2$ ramification over $z=1$.

Part~(b): set $F(x)=f(\phi(x))$, so that
\[
y(x)=\frac{F(x)}{\sqrt{1-4x}}.
\]
The Gauss equation for $f$ is
\[
z(1-z)f''+\left(1-\frac32 z\right)f'-\frac1{18}f=0.
\]
By the chain rule,
\[
F'=\phi' f'(\phi),\qquad
F''=(\phi')^2 f''(\phi)+\phi'' f'(\phi),
\]
hence
\[
f'(\phi)=\frac{F'}{\phi'},
\qquad
f''(\phi)=\frac{1}{(\phi')^2}\left(F''-\frac{\phi''}{\phi'}F'\right).
\]
Substituting $z=\phi(x)$ and these expressions into the Gauss equation gives
\[
\phi(1-\phi)\left(F''-\frac{\phi''}{\phi'}F'\right)
+\left(1-\frac32\phi\right)\phi'F'
-\frac1{18}(\phi')^2F=0.
\]
Now
\[
\phi(x)=\frac{108x^2}{(1-4x)^3},\qquad
\phi'(x)=\frac{216x(1+2x)}{(1-4x)^4},\qquad
\phi''(x)=\frac{216(1+16x+16x^2)}{(1-4x)^5},
\]
and
\[
1-\phi(x)=\frac{(1+2x)^2(1-16x)}{(1-4x)^3}.
\]
Substituting these formulas and multiplying by $(1-4x)^8/\bigl(108x(1+2x)^2\bigr)$ yields
\[
x(1-4x)^2(1-16x)F''+(1-4x)(64x^2-26x+1)F'-24xF=0.
\]
Thus the pulled-back equation already has singularities only at
$x=0,\,\tfrac1{16},\,\tfrac14,\,\infty$.
Now write $F(x)=s(x)\,y(x)$ with $s(x)=\sqrt{1-4x}$.
Then $s'/s=-2/(1-4x)$ and $s''/s=-4/(1-4x)^2$.
Substituting $F=sy$ into the previous equation and simplifying, the coefficient of $y'$ becomes
$(1-4x)(1-30x+128x^2)$,
and the coefficient of $y$ becomes
$(1-4x)\bigl(-2(1-8x)\bigr)$.
Therefore
\[
(1-4x)\Bigl(x(1-4x)(1-16x)y''+(1-30x+128x^2)y'-2(1-8x)y\Bigr)=0.
\]
Dividing by $1-4x$ gives~\eqref{eq:heun-domb}.

Part~(c): since $G=y^2$, the Leibniz-rule computation of Theorem~\ref{thm:theta-bridge} (applied to the Heun operator~\eqref{eq:heun-domb} in place of the Gauss operator) produces a third-order ODE for~$G$. Explicitly: one expresses $\theta_x(y^2)$, $\theta_x(y\theta_x y)$, $\theta_x((\theta_x y)^2)$ using~\eqref{eq:heun-domb}, converts to Euler form, and collects by powers of~$x$. The result is~\eqref{eq:domb-ode-theta}. Extracting the coefficient of $x^n$ gives $(n+1)^3D_{n+1}=2(2n+1)(5n^2+5n+2)D_n-64n^3D_{n-1}$, which is the standard Domb recurrence. As a numerical sanity check, $D_0,\ldots,D_7=1,4,28,256,2716,31504,387136,4951552$.

Part~(d) is a restatement of parts~(a)--(c).
\end{proof}

\begin{theorem}[pullback--twist functoriality for CMFs]\label{thm:pullback-cmf}
Let $B(z,\boldsymbol\lambda)$ be a row basis generating a rank-$r$ CMF with shift matrices $M_v(z,\boldsymbol\lambda)$ and differential component $M_{\theta_z}(z,\boldsymbol\lambda)$. Let $\phi\colon\mathbb P^1\to\mathbb P^1$ be a nonconstant rational map and $\rho(x,\boldsymbol\lambda)$ a nonzero rational function of~$x$ (possibly depending on parameters $\boldsymbol\lambda$). Define
\[
\widetilde B(x,\boldsymbol\lambda)=\rho(x,\boldsymbol\lambda)\,B\bigl(\phi(x),\boldsymbol\lambda\bigr).
\]
Then $\widetilde B$ generates a CMF with
\begin{equation}\label{eq:pullback-shift}
\widetilde M_v(x,\boldsymbol\lambda)=\frac{\sigma_v(\rho)}{\rho}\,M_v\bigl(\phi(x),\boldsymbol\lambda\bigr),
\end{equation}
\begin{equation}\label{eq:pullback-diff}
\widetilde M_{\theta_x}(x,\boldsymbol\lambda)=\frac{x\rho'}{\rho}\,I_r+\frac{x\phi'}{\phi}\,M_{\theta_z}\bigl(\phi(x),\boldsymbol\lambda\bigr).
\end{equation}
If $\rho$ does not depend on $\boldsymbol\lambda$, then the shift matrices transfer by simple composition: $\widetilde M_v(x)=M_v(\phi(x))$. Moreover, when $r=2$ and $\rho^{1/2}$ denotes a scalar square root of $\rho$, $\Sym^2$ converts the rank-$2$ twist $\rho^{1/2}$ into the rank-$3$ twist $\rho$:
\begin{equation}\label{eq:sym2-commutes}
\Sym^2\bigl(\rho^{1/2}\,B\circ\phi\bigr)=\rho\,\Sym^2(B)\circ\phi.
\end{equation}
\end{theorem}

\begin{proof}
For the shift matrices: $\sigma_v(\widetilde B)=\sigma_v(\rho)\,\sigma_v(B)(\phi(x))=\sigma_v(\rho)\,B(\phi(x))\,M_v(\phi(x))=\widetilde B\cdot(\rho^{-1}\sigma_v(\rho))\,M_v(\phi(x))$.

For the differential component: $\theta_x(\widetilde B)=x\rho'\,B(\phi(x))+\rho\,x\phi'\,B'(\phi(x))$. Since $B'(\phi(x))=B(\phi(x))\,M_{\theta_z}(\phi(x))/\phi(x)$ (because $\theta_z B=B M_{\theta_z}$ and $\theta_z=z\,d/dz$), we get $\theta_x(\widetilde B)=\widetilde B\bigl((x\rho'/\rho)I_r+(x\phi'/\phi)M_{\theta_z}(\phi(x))\bigr)$.

For the commutativity: if $\widetilde b=\rho^{1/2}b(\phi(x))$ for a rank-$2$ basis element $b$, then $\widetilde b^2=\rho\,b(\phi(x))^2$, and analogously for $\widetilde b_1\widetilde b_2$, so the $\Sym^2$ basis transforms as $\rho\cdot(\Sym^2 B)\circ\phi$.
\end{proof}

\begin{remark}[unified picture]\label{rem:unified}
With Theorems~\ref{thm:domb-pullback} and~\ref{thm:pullback-cmf}, all three printed order-$3$ kernels are covered by a single mechanism: $\Sym^2$ of a rank-$2$ hypergeometric object, composed (when necessary) with a Belyi pullback and algebraic twist, followed by summation. Table~\ref{tab:kernels-v6} summarizes the complete picture.
\end{remark}

\begin{table}[ht]
\centering
\caption{Complete identification of the three printed order-$3$ kernels.}\label{tab:kernels-v6}
\smallskip
\begin{tabular}{lcccc}
\hline
Kernel & OEIS & AZ case & Gauss params & Pullback \\
\hline
First $\pi$ & $A036917$ & $(\beta)$ & $(\tfrac12,\tfrac12,1)$ & identity \\
Catalan & non-sporadic & --- & $(\tfrac12,1,\tfrac32)$ & identity \\
Second $\pi$ & $A002895$ (Domb) & $(\alpha)$ & $(\tfrac16,\tfrac13,1)$ & $108x^2/(1{-}4x)^3$ \\
\hline
\end{tabular}
\end{table}

\section{Integer sequences from the Belyi pullback scan}\label{sec:belyi-scan}

\subsection{Setup and the output list}

Besides the three printed order-$3$ kernels discussed in Sections~\ref{sec:summation}--\ref{sec:aperylike-clausen}, the Belyi-pullback scan also produces a small set of further integer sequences. Concretely, we scanned
\[
[x^n]\;\lambda^n\,\Ftwoone(a,b;c;\phi(x))^2
\]
over $5040$ tuples $(a,b,c,\phi,u,\lambda)$ and retained those rows for which the scaled coefficients were integral for $0\le n\le 19$ in exact rational arithmetic. The full scan data are archived at \url{https://doi.org/10.5281/zenodo.19286066}.

For the $11$ rows listed in Table~\ref{tab:belyi-scan-11}, we write
\[
A(x):=\sum_{n\ge 0} a_n x^n
=
\Ftwoone(a,b;c;\phi(\lambda x))^2
=
\sum_{n\ge 0}\lambda^n [x^n]\Ftwoone(a,b;c;\phi(x))^2\,x^n.
\]
The row numbers \(\#2,\#3,\#4,\#5,\#9,\#10,\#11,\#12,\#13,\#14,\#15\) are the scan labels from the dataset.

\begin{table}[H]
\centering
\caption{The $11$ additional integer sequences found in the Belyi-pullback scan. Here \(a_n=[x^n]\Ftwoone(a,b;c;\phi(\lambda x))^2\), and the OEIS status is the result of a web lookup on 2026-03-29.}
\label{tab:belyi-scan-11}
\smallskip
\scriptsize
\renewcommand{\arraystretch}{1.15}
\begin{tabular}{@{}r l l c p{7.5cm} l@{}}
\toprule
\# & $(a,b,c)$ & $\phi(x)$ & $\lambda$ & first $8$ terms $(a_0,\dots,a_7)$ & OEIS \\
\midrule
2 & \(\bigl(\tfrac12,\tfrac12,2\bigr)\) & \(\dfrac{4x}{(1-x)^2}\) & \(4\) &
1, 4, 60, 888, 13960, 231904, 4025904, 72372528 & not found \\
3 & \(\bigl(\tfrac16,\tfrac23,\tfrac32\bigr)\) & \(\dfrac{27x}{(1-4x)^3}\) & \(1\) &
1, 4, 94, 2196, 56061, 1535040, 44202600, 1321014672 & not found \\
4 & \(\bigl(\tfrac13,\tfrac13,1\bigr)\) & \(\dfrac{27x}{(1-4x)^3}\) & \(1\) &
1, 6, 153, 3912, 108042, 3161196, 96340410, 3024934080 & not found \\
5 & \(\bigl(\tfrac16,\tfrac12,\tfrac12\bigr)\) & \(\dfrac{27x}{(1-4x)^3}\) & \(1\) &
1, 9, 270, 8154, 259209, 8529921, 287329140, 9841383288 & not found \\
9 & \(\bigl(\tfrac14,\tfrac34,\tfrac12\bigr)\) & \(\dfrac{4x}{(1-x)^2}\) & \(1\) &
1, 3, 17, 95, 537, 3059, 17513, 100607 & not found \\
10 & \(\bigl(\tfrac13,\tfrac23,\tfrac32\bigr)\) & \(\dfrac{27x}{(1-4x)^3}\) & \(1\) &
1, 8, 208, 5376, 148480, 4317184, 130351104, 4049600512 & not found \\
11 & \(\bigl(\tfrac14,\tfrac34,\tfrac32\bigr)\) & \(\dfrac{27x}{(1-4x)^3}\) & \(4\) &
1, 27, 2754, 279855, 30556062, 3525880590, 423488705220, 52412646653559 & not found \\
12 & \(\bigl(\tfrac13,\tfrac23,1\bigr)\) & \(\dfrac{27x}{(1-4x)^3}\) & \(1\) &
1, 12, 360, 10776, 337656, 10931616, 362216088, 12210185424 & not found \\
13 & \(\bigl(\tfrac13,\tfrac12,\tfrac12\bigr)\) & \(\dfrac{27x}{(1-4x)^3}\) & \(1\) &
1, 18, 621, 21168, 738090, 26128764, 934657434, 33688028808 & not found \\
14 & \(\bigl(\tfrac23,\tfrac23,1\bigr)\) & \(\dfrac{27x}{(1-4x)^3}\) & \(1\) &
1, 24, 882, 31560, 1138569, 41331312, 1507503024, 55190279616 & not found \\
15 & \(\bigl(\tfrac12,\tfrac12,1\bigr)\) & \(4x(1-x)\) & \(4\) &
1, 8, 56, 384, 2648, 18496, 131008, 940032 & not found \\
\bottomrule
\end{tabular}
\end{table}

\subsection{Integrality}

\begin{theorem}\label{thm:belyi-integrality}
For each of the rows \(\#2,\#3,\#4,\#5,\#9,\#10,\#11,\#12,\#13,\#14,\#15\), let
\[
G(x):=\Ftwoone(a,b;c;\phi(x))^2,
\qquad
A(x):=G(\lambda x)=\sum_{n\ge 0} a_n x^n,
\]
where \((a,b,c)\), \(\phi(x)\), and \(\lambda\) are those of Table~\ref{tab:belyi-scan-11}. Then
\[
A(x)\in \Z[[x]].
\]
Equivalently,
\[
a_n=\lambda^n[x^n]\Ftwoone(a,b;c;\phi(x))^2\in \Z
\qquad (n\ge 0).
\]
\end{theorem}

\begin{proof}
Write
\[
F(z)=\Ftwoone(a,b;c;z)=\sum_{m\ge 0} f_m z^m.
\]

We shall repeatedly use the following elementary observation: if \(\mu,\nu\in\Z\),
\(r\in\Z_{\ge 1}\), and \(\mu^m f_m\in\Z\) for all \(m\ge 0\), then
\[
F\!\left(\frac{\mu x}{(1-\nu x)^r}\right)
=\sum_{m\ge 0} (\mu^m f_m)\,x^m(1-\nu x)^{-rm}\in\Z[[x]],
\]
because each \((1-\nu x)^{-rm}\) lies in \(\Z[[x]]\).

We also use the following \(p\)-adic counting fact. For \(r\in\{1,2\}\) and \(N\ge 1\), put
\[
P_r(N):=\prod_{j=1}^N (3j-r).
\]
If \(p\neq 3\) is prime, then for each \(\ell\ge 1\) the congruence
\[
3j\equiv r \pmod{p^\ell}
\]
has exactly one residue class modulo \(p^\ell\), hence among \(j=1,\dots,N\) it occurs at
least \(\lfloor N/p^\ell\rfloor\) times. Therefore
\[
v_p(P_r(N))\ge \sum_{\ell\ge 1}\Bigl\lfloor \frac{N}{p^\ell}\Bigr\rfloor=v_p(N!).
\]
For \(p=3\) we only need the crude bound
\[
v_3(N!)=\sum_{\ell\ge 1}\Bigl\lfloor \frac{N}{3^\ell}\Bigr\rfloor\le \frac N2.
\]

We now treat the eleven rows.

\medskip
\noindent\textbf{Row \#2.}
Here
\[
f_m=\frac{(\frac12)_m^2}{(2)_m\,m!}.
\]
Using
\[
\Bigl(\frac12\Bigr)_m=\frac{(2m)!}{4^m m!},
\qquad
(2)_m=(m+1)!,
\]
we get
\[
16^m f_m
=16^m\frac{(\frac12)_m^2}{(2)_m\,m!}
=\frac{1}{m+1}\binom{2m}{m}^2\in\Z.
\]
Since
\[
\phi(4x)=\frac{16x}{(1-4x)^2},
\]
the observation above yields
\[
\Ftwoone\!\left(\tfrac12,\tfrac12;2;\phi(4x)\right)\in\Z[[x]].
\]
Squaring, we get \(A(x)\in\Z[[x]]\).

\medskip
\noindent\textbf{Row \#3.}
Here
\[
f_m=\frac{(\frac16)_m(\frac23)_m}{(\frac32)_m\,m!}.
\]
A direct simplification gives
\[
27^m f_m
=3^m\frac{\prod_{k=1}^m(6k-5)(6k-2)}{(2m+1)!}
=3^m\frac{\prod_{j=1}^{2m}(3j-2)}{(2m+1)!}.
\]
Let \(p\neq 3\) and \(\ell\ge 1\). The congruence \(3j\equiv 2\pmod{p^\ell}\) defines a
nonzero residue class modulo \(p^\ell\). Hence the number of solutions among
\(j=1,\dots,2m\) equals the number of solutions among \(j=0,\dots,2m\), and the latter is
at least \(\lfloor(2m+1)/p^\ell\rfloor\). Therefore
\[
v_p\!\left(\prod_{j=1}^{2m}(3j-2)\right)\ge
\sum_{\ell\ge 1}\Bigl\lfloor\frac{2m+1}{p^\ell}\Bigr\rfloor
=
v_p((2m+1)!).
\]
For \(p=3\) we have
\[
v_3((2m+1)!)
<
\sum_{\ell\ge 1}\frac{2m+1}{3^\ell}
=
\frac{2m+1}{2}
<m+1,
\]
hence \(v_3((2m+1)!)\le m\). So \(27^m f_m\in\Z\) for all \(m\).
Since
\[
\phi(x)=\frac{27x}{(1-4x)^3},
\]
it follows that
\[
\Ftwoone\!\left(\tfrac16,\tfrac23;\tfrac32;\phi(x)\right)\in\Z[[x]],
\]
and therefore \(A(x)\in\Z[[x]]\).

\medskip
\noindent\textbf{Rows \#4 and \#14.}
For \#4,
\[
f_m=\frac{(\frac13)_m^2}{(m!)^2},
\qquad
27^m f_m
=
3^m\left(\frac{\prod_{j=1}^m(3j-2)}{m!}\right)^2.
\]
For \(p\neq 3\) the counting fact gives
\[
v_p\!\left(\prod_{j=1}^m(3j-2)\right)\ge v_p(m!),
\]
and for \(p=3\),
\[
2v_3(m!)\le 2\sum_{\ell\ge 1}\frac{m}{3^\ell}=m.
\]
Hence \(27^m f_m\in\Z\).

For \#14,
\[
f_m=\frac{(\frac23)_m^2}{(m!)^2},
\qquad
27^m f_m
=
3^m\left(\frac{\prod_{j=1}^m(3j-1)}{m!}\right)^2,
\]
and the same argument shows \(27^m f_m\in\Z\).

Since both rows have
\[
\phi(x)=\frac{27x}{(1-4x)^3},
\]
we obtain
\[
\Ftwoone\!\left(\tfrac13,\tfrac13;1;\phi(x)\right)\in\Z[[x]],
\qquad
\Ftwoone\!\left(\tfrac23,\tfrac23;1;\phi(x)\right)\in\Z[[x]],
\]
hence after squaring, \(A(x)\in\Z[[x]]\) in both cases.

\medskip
\noindent\textbf{Row \#10.}
Here
\[
f_m=\frac{(\frac13)_m(\frac23)_m}{(\frac32)_m\,m!}.
\]
Using
\[
\Bigl(\frac13\Bigr)_m\Bigl(\frac23\Bigr)_m=\frac{(3m)!}{27^m m!},
\qquad
\Bigl(\frac32\Bigr)_m=\frac{(2m+1)!}{4^m m!},
\]
we get
\[
27^m f_m
=
\frac{4^m(3m)!}{(2m+1)!\,m!}
=
\frac{4^m}{2m+1}\binom{3m}{m}.
\]
The number
\[
\frac{1}{2m+1}\binom{3m}{m}
=
\frac{1}{3m+1}\binom{3m+1}{m}
\]
is the Fuss--Catalan number of order \(3\), hence is an integer. Therefore
\(27^m f_m\in\Z\). Since
\[
\phi(x)=\frac{27x}{(1-4x)^3},
\]
we conclude that
\[
\Ftwoone\!\left(\tfrac13,\tfrac23;\tfrac32;\phi(x)\right)\in\Z[[x]],
\]
and hence \(A(x)\in\Z[[x]]\).

\medskip
\noindent\textbf{Row \#12.}
Here
\[
f_m=\frac{(\frac13)_m(\frac23)_m}{(1)_m\,m!},
\]
so
\[
27^m f_m=\frac{(3m)!}{(m!)^3}\in\Z.
\]
Again
\[
\phi(x)=\frac{27x}{(1-4x)^3},
\]
hence
\[
\Ftwoone\!\left(\tfrac13,\tfrac23;1;\phi(x)\right)\in\Z[[x]],
\]
and therefore \(A(x)\in\Z[[x]]\).

\medskip
\noindent\textbf{Rows \#5 and \#13.}
In \#5 we have \(c=b=\tfrac12\), hence
\[
\Ftwoone\!\left(\tfrac16,\tfrac12;\tfrac12;z\right)=(1-z)^{-1/6},
\]
so
\[
G(x)=\left(1-\phi(x)\right)^{-1/3}.
\]
Thus
\[
G(x)=\sum_{m\ge 0}\frac{(\frac13)_m}{m!}\,\phi(x)^m.
\]
Now
\[
27^m\frac{(\frac13)_m}{m!}
=
9^m\frac{\prod_{j=1}^m(3j-2)}{m!}.
\]
For \(p\neq 3\) the counting fact gives integrality, and for \(p=3\) we have
\(v_3(m!)\le m/2<2m\). Hence
\[
27^m\frac{(\frac13)_m}{m!}\in\Z.
\]
Therefore \(G(x)\in\Z[[x]]\), i.e.\ \(A(x)\in\Z[[x]]\).

Similarly, in \#13 we have \(c=b=\tfrac12\), so
\[
\Ftwoone\!\left(\tfrac13,\tfrac12;\tfrac12;z\right)=(1-z)^{-1/3},
\]
hence
\[
G(x)=\left(1-\phi(x)\right)^{-2/3}
=
\sum_{m\ge 0}\frac{(\frac23)_m}{m!}\,\phi(x)^m.
\]
And
\[
27^m\frac{(\frac23)_m}{m!}
=
9^m\frac{\prod_{j=1}^m(3j-1)}{m!}\in\Z
\]
by the same argument. Thus \(A(x)\in\Z[[x]]\).

\medskip
\noindent\textbf{Row \#9.}
The quadratic Schwarz identity gives
\[
\Ftwoone\!\left(\tfrac14,\tfrac34;\tfrac12;z\right)^2
=
\frac12\left((1-z)^{-1}+(1-z)^{-1/2}\right).
\]
Let
\[
u=\phi(x)=\frac{4x}{(1-x)^2}.
\]
Then
\[
A(x)=G(x)=\frac12\left((1-u)^{-1}+(1-u)^{-1/2}\right).
\]
Now
\[
(1-u)^{-1}
=
\sum_{m\ge 0} u^m
=
\sum_{m\ge 0} 4^m x^m(1-x)^{-2m}\in\Z[[x]],
\]
and
\[
(1-u)^{-1/2}
=
\sum_{m\ge 0}\binom{2m}{m}\left(\frac{u}{4}\right)^m
=
\sum_{m\ge 0}\binom{2m}{m}x^m(1-x)^{-2m}\in\Z[[x]].
\]
Moreover, for every \(m\ge 1\), both \(4^m\) and \(\binom{2m}{m}\) are even, so the
coefficient of \(x^n\) is even in both series for all \(n\ge 1\), and the constant term of
their sum is \(2\). Hence the sum lies in \(2\Z[[x]]\), and therefore
\[
A(x)\in\Z[[x]].
\]

\medskip
\noindent\textbf{Row \#11.}
Again by a quadratic Schwarz identity,
\[
\Ftwoone\!\left(\tfrac14,\tfrac34;\tfrac32;z\right)^2
=
\frac{2}{1+\sqrt{1-z}}.
\]
Using the Catalan generating series,
\[
\frac{2}{1+\sqrt{1-z}}
=
\sum_{m\ge 0} C_m\left(\frac{z}{4}\right)^m,
\qquad
C_m=\frac{1}{m+1}\binom{2m}{m}\in\Z.
\]
Now
\[
\phi(4x)=\frac{108x}{(1-16x)^3},
\]
hence
\[
A(x)=G(4x)
=
\sum_{m\ge 0} C_m\left(\frac{27x}{(1-16x)^3}\right)^m.
\]
Each summand lies in \(\Z[[x]]\), therefore \(A(x)\in\Z[[x]]\).

\medskip
\noindent\textbf{Row \#15.}
Let
\[
u=\phi(4x)=16x(1-4x).
\]
The Clausen identity in the form
\[
\Ftwoone\!\left(\tfrac12,\tfrac12;1;u\right)^2
=
\FthreeTwo\!\left(\tfrac12,\tfrac12,\tfrac12;1,1;4u(1-u)\right)
\]
gives
\[
A(x)=G(4x)
=
\FthreeTwo\!\left(\tfrac12,\tfrac12,\tfrac12;1,1;
64x(1-4x)(1-8x)^2\right).
\]
Write
\[
\FthreeTwo\!\left(\tfrac12,\tfrac12,\tfrac12;1,1;t\right)
=
\sum_{m\ge 0} c_m t^m,
\qquad
c_m=\frac{(\frac12)_m^3}{(m!)^3}.
\]
Then
\[
64^m c_m=\binom{2m}{m}^3\in\Z.
\]
Since
\[
t=64x(1-4x)(1-8x)^2,
\]
each term \(c_m t^m\) has integer coefficients. Therefore \(A(x)\in\Z[[x]]\).

This completes the proof for all eleven rows.
\end{proof}

\subsection{Symmetric-square classification}

\begin{proposition}[Sym$^2$ classification of the scan output]\label{prop:belyi-sym2}
For each row of Table~\ref{tab:belyi-scan-11}, the generating function
\[
A(x)=\Ftwoone(a,b;c;\phi(\lambda x))^2
\]
is obtained from a rank-$2$ Gauss hypergeometric object by taking \(\Sym^2\) and then applying the rational pullback \(z=\phi(\lambda x)\).

More precisely:
\[
\text{rows \#5,\#13 admit annihilating operators of order }1,
\]
\[
\text{rows \#9,\#11 admit annihilating operators of order at most }2,
\]
\[
\text{rows \#2,\#3,\#4,\#10,\#12,\#14,\#15 admit annihilating operators of order at most }3.
\]
Thus the $11$ rows fall into three classes: binomial collapse, quadratic algebraic, and order-$3$ pullback cases.
\end{proposition}

\begin{proof}
The first assertion is tautological: if
\[
f(x)=\Ftwoone(a,b;c;\phi(\lambda x)),
\]
then \(A(x)=f(x)^2\), so \(A(x)\) belongs to the symmetric square of the rank-$2$ differential module generated by \(f\). Thus every row is a \(\Sym^2\)-pullback of a Gauss hypergeometric object.

For rows \#5 and \#13, the explicit formulas already recorded above are
\[
A_5(x)=\frac{1-4x}{\bigl(1-39x+48x^2-64x^3\bigr)^{1/3}},
\qquad
A_{13}(x)=\frac{(1-4x)^2}{\bigl(1-39x+48x^2-64x^3\bigr)^{2/3}}.
\]
Hence \(A_5'(x)/A_5(x)\) and \(A_{13}'(x)/A_{13}(x)\) are rational functions, so both
series satisfy first-order linear differential equations over \(\Q(x)\).

For rows \#9 and \#11, put \(u=4x/(1-x)^2\) and \(u=108x/(1-16x)^3\) respectively. The Schwarz identities give
\[
A_9(x)=\frac12\left((1-u)^{-1}+(1-u)^{-1/2}\right),
\qquad
A_{11}(x)=\frac{2}{1+\sqrt{1-u}},
\]
so in either case \(A(x)\in \Q(x,\sqrt{1-u})\). If
\(V=\Q(x)\oplus \Q(x)\sqrt{1-u}\),
then \(V\) is stable under \(d/dx\), because
\[
\frac{d}{dx}\sqrt{1-u}=-\frac{u'}{2(1-u)}\sqrt{1-u}\in V.
\]
Therefore \(A(x),A'(x),A''(x)\in V\), a two-dimensional \(\Q(x)\)-vector space. Hence these three elements are linearly dependent over \(\Q(x)\), so rows \#9 and \#11 admit annihilating operators of order at most~\(2\).

For rows \#2,\#3,\#4,\#10,\#12,\#14, let
\(g(z)=\Ftwoone(a,b;c;z)^2\).
By Theorem~\ref{thm:theta-bridge}, \(g\) is annihilated by a linear differential operator of order \(3\) with coefficients in \(\Q(z)\). Replacing \(z\) by the rational function \(\phi(\lambda x)\) and applying the chain rule produces, after clearing denominators, a linear differential operator of order at most \(3\) with coefficients in \(\Q(x)\) annihilating
\(A(x)=g(\phi(\lambda x))\).

For row \#15, the identity already recorded above is
\[
A_{15}(x)=\FthreeTwo\!\left(\tfrac12,\tfrac12,\tfrac12;1,1;64x(1-4x)(1-8x)^2\right).
\]
The standard hypergeometric differential equation for
\(H(t)=\FthreeTwo(\tfrac12,\tfrac12,\tfrac12;1,1;t)\)
has order \(3\). After the rational pullback \(t=64x(1-4x)(1-8x)^2\), the resulting series \(A_{15}(x)\) is therefore annihilated by an operator of order at most \(3\) over \(\Q(x)\).

This proves the stated order bounds for all eleven rows.
\end{proof}

\begin{remark}[no low-degree order-$2$ recurrence for the raw scan sequences]\label{rem:belyi-not-apery}
Let
\[
A(x)=\sum_{n\ge 0} a_n x^n
\]
be the generating function attached to one row of Table~\ref{tab:belyi-scan-11}. None of the $11$ raw coefficient sequences \((a_n)_{n\ge 0}\) satisfies a polynomial recurrence of order \(2\)
\[
P_0(n)a_n+P_1(n)a_{n+1}+P_2(n)a_{n+2}=0
\]
with \(\deg P_j\le 12\). Indeed, writing
\[
P_j(n)=\sum_{d=0}^{12} p_{j,d}n^d
\]
gives a homogeneous linear system in the \(39\) unknown coefficients \(p_{j,d}\). Using the equations for \(n=0,1,\dots,38\) yields a \(39\times39\) matrix. For the \(11\) sequences of Table~\ref{tab:belyi-scan-11}, this matrix has nonzero determinant modulo \(1000003\); in scan order \(\#2,\#3,\#4,\#5,\#9,\#10,\#11,\#12,\#13,\#14,\#15\), the residues are
\[
881437,\ 261488,\ 271237,\ 594271,\ 784564,\ 945202,\ 14170,\ 814591,\ 311594,\ 572039,\ 505664.
\]
Therefore no such low-degree order-$2$ polynomial recurrence exists for any of the raw row sequences. This rules out direct coincidence with the standard Zagier/Almkvist--Zudilin/Cooper kernels, and more generally with transformations that preserve both recurrence order and the degree bound \(\le 12\). We do not claim here to exclude all possible gauge or rescaling transforms.
\end{remark}

\section{Discussion and open problems}\label{sec:discussion}
The results above separate three levels of structure.

Section~\ref{sec:belyi-scan} adds a second concrete output channel to the story. Beyond the three printed order-$3$ kernels, the same Belyi-pullback machinery yields $11$ additional integer sequences from the $5040$-tuple scan. Theorem~\ref{thm:belyi-integrality} proves integrality for all eleven rows, Proposition~\ref{prop:belyi-sym2} shows that they are again \(\Sym^2\)-pullbacks of rank-$2$ Gauss objects, and Remark~\ref{rem:belyi-not-apery} shows that none of them hides an order-$2$ Ap\'ery-like recurrence of low degree.

At the \emph{formula level}, Raz et al.\ correctly canonicalize partial sums or convergents and find minimal polynomial recurrences \cite{Raz}. In that sense, the printed examples in Appendix~B.6 are indeed order-$3$ formulas. At the \emph{kernel level}, however, Theorem~\ref{thm:three-printed-factor} shows that every order-$3$ recurrence explicitly printed in Appendix~B.6 is obtained from an order-$2$ kernel by one summation. This is a structural refinement of the public order-$3$ examples, not a correction of the formula-level minimal orders in \cite{Raz}.

Theorem~\ref{thm:A036917kernel} identifies the first printed \(\pi\)-kernel with $A036917$; Theorem~\ref{thm:domb} identifies the second with the Domb numbers $A002895$; and Theorem~\ref{thm:catalan-square} traces the Catalan kernel to the Gauss-square point $(\tfrac12,1,\tfrac32)$. Theorem~\ref{thm:domb-pullback} then closes the Domb branch by exhibiting the Domb generating function as a Belyi-pulled-back twisted Gauss square at $(\tfrac16,\tfrac13,1)$, and Theorem~\ref{thm:pullback-cmf} establishes that pullback--twist preserves CMF structure. Consequently, all three printed order-$3$ kernels are covered by a single mechanism (Table~\ref{tab:kernels-v6}):
\[
\text{rank-}2\ {}_2F_1\text{-object} \xrightarrow{\;\Sym^2\;} \text{rank-}3\text{ module} \xrightarrow{\;\text{pullback/twist}\;} \text{kernel} \xrightarrow{\;\text{summation}\;} \text{order-}3\text{ formula.}
\]

Theorem~\ref{thm:inverse} provides an inverse classification. For a fixed $\Sym^2$-type Riemann scheme, the local exponents do not determine the operator uniquely---one accessory parameter $\lambda$ remains free. The $\Sym^2(\mathrm{Gauss})$ locus is cut out by the single equation $\lambda=2\gamma_1\gamma_2(1-2\alpha)$, which recovers $(a,b,c)$ explicitly. This is not an algorithmic detection (cf.\ Singer \cite{Singer} and van Hoeij \cite{vanHoeij} for operator-level symmetric-square tests) but a closed-form criterion on the accessory parameter. Its verification on the $A036917$ and Catalan cases (Remark~\ref{rem:inverse-check}) confirms the consistency of the direct and inverse approaches.

We emphasize that the individual links in this chain are not new: the third-order ODE for $f^2$ is due to Chaundy \cite{Chaundy} and Vid\=unas \cite{Vidunas}; the connection between Ap\'ery-like sequences and symmetric squares of hypergeometric equations has been understood since Almkvist, van~Straten, and Zudilin \cite{AlmkvistVSZ}, and Gorodetsky \cite{Gorodetsky} states explicitly that the generating functions of the Almkvist--Zudilin sequences are essentially the squares of the corresponding Zagier sequences; the connection between $A036917$ and $\pi$-series is established by Chan--Verrill \cite{ChanVerrill} and Cooper \cite{Cooper}; and the coefficient recurrence for general parameters is now in Mao--Tian \cite{MaoTian}. The Belyi map $\phi(x)=108x^2/(1-4x)^3$ and the Domb representation through a pulled-back Gauss square are known in the modular-forms literature \cite{ChanChanLiu,ChanZudilin}; the underlying pullback of differential equations is a standard construction \cite{Vidunas}. Our contribution is threefold: the unified CMF reinterpretation of these classical links, the pullback--twist functoriality theorem for CMFs (Theorem~\ref{thm:pullback-cmf}, an elementary restatement of standard gauge transport in the CMF setting), and the inverse classification via the accessory parameter (Theorem~\ref{thm:inverse}).

Several concrete problems remain open.
\begin{enumerate}[label=\textup{(\alph*)}]
\item Raz et al.\ report four order-$3$ canonical forms for \(\pi\), but only two explicit order-$3$ \(\pi\)-recurrences are printed in the public Appendix~B.6. Determining whether the fourth public-\emph{unprinted} order-$3$ \(\pi\)-canonical form is also a summation lift is the most immediate missing computation.

\item Theorems~\ref{thm:theta-bridge}, \ref{thm:domb-pullback}, and~\ref{thm:pullback-cmf} show that the printed kernels are intrinsic to the \emph{differential} components of the ambient Gauss CMF and its pullback. It remains open whether any of them also appear as genuine \emph{trajectories} of the current rank-$2$ \(\pi\)-CMF or of its square.

\item Theorem~\ref{thm:pullback-cmf} handles rational pullbacks and scalar twists. Extending this to algebraic pullbacks with algebraic twists in full generality---and classifying which Belyi maps are compatible with Ap\'ery-like integrality---would connect the present framework to the modular parametrizations underlying the sporadic sequences.

\item More generally, one may ask for a classification of parameters \((a,b,c)\) and Belyi maps $\phi$ for which the coefficient sequence of $u(x)\cdot{}_2F_1(a,b;c;\phi(x))^2$ is an Ap\'ery-like sequence, and for which such kernels can be realized by CMF trajectories.

\item The $11$ sequences of Section~\ref{sec:belyi-scan} raise a geometric and arithmetic problem of their own: identify the corresponding motives or period interpretations, determine the dessins attached to the Belyi maps \(4x/(1-x)^2\) and \(27x/(1-4x)^3\), and understand the associated \(L\)-functions.

\item It is also natural to ask whether the same $11$ sequences satisfy Dwork-type congruences or stronger supercongruences at primes \(p\), and whether these congruences can be read off directly from the pullback data \((a,b,c,\phi,\lambda)\).

\item Mao and Tian \cite{MaoTian} derive the coefficient recurrence for \({}_2F_1(a,b;c;z)^k\) for \(k=2\) and \(k=3\). It would be interesting to extend this to general \(k\) and to determine whether the corresponding \(\Sym^k\)(CMF) construction yields further Ap\'ery-like kernels for \(k\ge 3\).
\end{enumerate}

\end{document}